\documentclass{amsart}
\usepackage{amsmath, amssymb, euscript,  enumerate}

\newtheorem{thm}{Theorem}[section]
\newtheorem{cor}[thm]{Corollary}
\newtheorem{lem}[thm]{Lemma}
\newtheorem{claim}[thm]{Claim}
\newtheorem{remark}[thm]{Remark}
\newtheorem{prop}[thm]{Proposition}
\newtheorem{defn}[thm]{Definition}
\theoremstyle{remark}
\newtheorem{rem}[thm]{Remark}
\numberwithin{equation}{section}
\theoremstyle{plain}

\def\no{\noindent} 

\newcommand{\R}{\mathbb R}

\newcommand{\A}{\mathcal A}

\newcommand{\D}{\mathcal D}

\newcommand{\Si}{\mathcal S}
\newcommand{\F}{\mathcal F}
\newcommand{\e}{\epsilon}

\newcommand{\dist}{\mathrm{dist}}

\newcommand{\diam}{\mathrm{diam}}

\newcommand{\Int}{\mathrm{Int}}
\newcommand{\Sit}{\Sigma^i_t}
\newcommand{\Sii}{\Sigma^i}

\begin{document}

\title{On the evolution of  convex hypersurfaces \\
by the $Q_k$ flow}

\author{M. Cristina Caputo$^*$}

\address{Department of
Mathematics, University Texas at Austin, Texas,
 USA}
\email{caputo@math.utexas.edu}

\author{Panagiota Daskalopoulos$^{**}$}

\address{Department of
Mathematics, Columbia University, New York,
 USA}
\email{pdaskalo@math.columbia.edu}

\author{Natasa Sesum$^{***}$}
\address{Department of Mathematics, Columbia University, New York,
USA}
\email{natasas@math.columbia.edu}

\thanks{$*:$ Partially supported
by NSF grant  0354639}
\thanks{$**:$ Partially supported
by NSF grant  0701045}
\thanks{$***:$ Partially supported
by NSF grant 0604657}

\begin{abstract}
We prove  the
existence and uniqueness of a $C^{1,1}$ solution of the $Q_k$ flow in the  viscosity sense for  compact  convex hypersurfaces 
$\Sigma_t$ embedded in $R^{n+1}$ ($n \geq 2$) . 
In particular,  for compact convex hypersurfaces with flat sides we show that,  under a certain non-degeneracy 
 initial condition,
 the interface
separating the flat from the strictly convex side, becomes smooth, and it
moves by the $Q_{k-1}$ flow at least for a short time.

\end{abstract}

\maketitle

\section{Introduction}

We consider, in this paper, the evolution of a compact  convex hypersurface 
$\Sigma_t$ embedded in $R^{n+1}$ ($n \geq 2$)  by the $Q_k-$flow for $1\leq k\leq n$, namely the equation
$$\frac{{\bf \partial P}}{\partial t} = - Q_k \, {\bf  \nu}$$
where each point $\bf P$ of the surface $\Sigma_t$ moves in the direction of its 
inner normal vector $\nu$
by a speed which is equal to the quotient 
$$Q_k(\lambda) =
\frac{S_k^n(\lambda)}{S_{k-1}^n(\lambda)}$$  
of   successive elementary symmetric polynomials of the principal curvatures $\lambda
= (\lambda_1, \dots, \lambda_n)$ of $\Sigma_t$. We recall that 
$$S_k^n(\lambda) = \sum_{1\leq i_1 < \cdots < i_k \leq n} \lambda_{i_1} \cdots \lambda_{i_k}.$$

Given a compact   convex hypersurface $\Sigma_0$ embedded  in  $\R^{n+1}$ 
($n \geq 2$), we let    
$F_0: M^n\to\R^{n+1}$ be  the immersion 
defining  $\Sigma_0$. We look for one parameter family of immersions 
$F:M^n\times [0,T)\to \R^{n+1}$ satisfying
\begin{equation}\label{eqn-qk}
\begin{cases}
\frac{\partial}{\partial t}F(p,t) &= -Q_k(p,t) \cdot\nu(p,t), \quad  \forall p\in M^n, t\ge 0, \nonumber \\
\,\,\, F(\cdot,0) &= F_0
\end{cases}
\tag{$*_k$}
\end{equation}
where $\nu$ is the outer unit normal and $Q_k(\lambda)$ is a  quotient of successive
elementary symmetric polynomials of the principal curvatures of $\Sigma_t = F(M^n,t)$. 

In
\cite{An1} Andrews proved that for any strictly convex hypersurface in
$\mathbb{R}^{n+1}$ the solution to (\ref{eqn-qk}) exists up to some 
finite time $T$ at which it shrinks to a point in an asymptotically
spherical manner. In \cite{D} Dieter considered smooth, convex
hypersurfaces in $\mathbb{R}^{n+1}$ with $S_{k-1}^n(\lambda) > 0$. Using
that condition she constructed cylindrically symmetric barriers to
prove that instantenously at time $t > 0$ the speed has a uniform
lower bound from below, and therefore the flow becomes strictly
parabolic.

\medskip

In this article we will consider  convex hypersurfaces, with no
assumption on $S_{k-1}^n(\lambda)$. In the first part of the paper we will  establish  the
existence and uniqueness of a $C^{1,1}$ solution of (\ref{eqn-qk}) in the  viscosity sense.  We
give  the following definition of a viscosity solution suggested by Andrews in \cite{An2}.

\begin{defn}
\label{def-visc}
A family of convex regions $\{\Omega_t\}_{0 <t<T}$ is called a {\it
viscosity solution} of (\ref{eqn-qk}) if the following
holds. For any smooth, strictly convex hypersurface $N_0$ contained
in $\Omega_{t_0}$, for some $t_0\in (0,T)$, the hypersurfaces $N_t$
given by solving (\ref{eqn-qk}) are contained in $\Omega_{t_0+t}$,
for all $t\in [0,T-t_0)$ in the time interval  of existence of
$N_t$. Furtermore, for any smooth, strictly convex hypersurface
$N_0$ which encloses $\Omega_{t_0}$, for some $t_0\in (0,T)$, the
hypersurfaces $N_t$ enclose $\Omega_{t_0+t}$,  for all $t\in
[0,T-t_0)$.
\end{defn}

Our first result states as follows:

\begin{thm}\label{thm-main1}
Let $\Sigma_0$ be a compact  convex  hypersurface in
$\mathbb{R}^{n+1}$ which is of class $C^{1,1}$. Then, there exists  a unique viscosity solution 
of  the \eqref{eqn-qk} flow  which is of class $C^{1,1}$. The solution exists 
up to the  time $T < \infty$ at which the enclosed volume becomes zero.
\end{thm}

In the second part of the paper we will consider the special case where the initial surface $\Sigma_0$
has flat sides. We will assume for simplicity that $\Sigma_0$ has only one flat side,  
 namely  $\Sigma_0
=\Sigma^1_0 \cup \Sigma^2_0$ with $\Sigma^1_0$ flat (i.e. it lies on a hyperplane in $\R^{n+1}$)
 and $\Sigma^2_0$  strictly convex. 
Since the equation is invariant under rotation, we may assume that $\Sigma^1_0$ lies on the
$x_{n+1}=0$ plane and that  that $\Sigma^2_0$ lies above this plane. 
Then, the lower part of the surface $\Sigma_0$ can be written
as the graph of a function 
$$x_{n+1}=u(x_1,\cdots,x_n)$$
over a compact domain $\varOmega \subset  {\R}^n$ containing
the initial flat side $\Sigma^1_0$.

We define
$$g:=\sqrt{u}$$
and  we call $g$ the {\em pressure} function. Let $\Gamma_0$ denote the boundary of the flat side $\Sigma^1_0$. Our main assumption on the initial   surface $\Sigma_0$
is that it is of class $C^{1,1}$, the function  $g$ is $C^2$ up to the flat side and it  satisfies the following non-degeneracy condition, which we
call {\it non-degeneracy condition~($\star$)}:

\begin{equation}
 | Dg | \geq \lambda
 \,\,\,\,\text{and}\,\,\,\,\,
(g_{ij})\geq \lambda, \qquad \mbox{on} \,\, \Gamma_0  \tag{$\star$}
\end{equation}
\noindent for some number  $\lambda >0$.\\
For each $1\leq i,j\leq  n-1$, $g_{ij}$ denote the second order
derivatives in the directions given by the vectors $\tau_i$ and $\tau_j$.  For each  $1\leq i\leq  n-1$, $\tau_i$ is defined so that the set $\mbox{Span} \, [\tau_1,\cdots,  \tau_{n-1}]$ is parallel to the tangent hyperplane to the level sets of $g$.

\begin{defn}\label{defn-surface}
We define $\mathfrak S$ to be the class of
 convex  compact hypersurfaces $\Sigma$ in ${\R}^{n+1}$ so that
$\Sigma = \Sigma^1 \cup \Sigma^2$, where $\Sigma^1$ is a surface contained in the hyperplane $x_{n+1}=0$ with smooth boundary  $\Gamma$, 
and $\Sigma^2$ is a strictly convex surface, smooth up to its  boundary $\Gamma$  
 which lies above the hyperplane $x_{n+1}=0$.
 \end{defn}

\begin{rem} Any initial surface  $\Sigma_0$ in the class $\mathfrak S$ is in particular a $C^{1,1}$
surface. Hence, by Theorem \ref{thm-main1},  there exists a unique $C^{1,1}$ solution 
$\Sigma_t$ of \eqref{eqn-qk} with initial data $\Sigma_0$.
\end{rem}

We will assume that $x_{n+1}=u(x_1,\cdots,x_n,t)$ defines the hypersurface $\Sigma_t$ near  the 
hyperplane $x_{n+1}=0$,  with $0\leq t\leq \tau$ for some short time
$\tau>0$. We will set $$g(\cdot,t)=\sqrt{u(\cdot,t)}.$$  Our goal is  to show the following result:

\begin{thm}\label{thm-main2}  Assume that at time
$t=0$, $\Sigma_0$ is a weakly convex  compact hypersurface in ${\R}^{n+1}$
which belongs to the class $\mathfrak S$ so  that the pressure function $g=\sqrt{u}$   is smooth up to the
interface $\Gamma_0$ and it satisfies the condition~($\star$). Let $\Sigma_t$ be the unique 
viscosity solution of \eqref{eqn-qk} for $2\leq k\leq n$ with initial data $\Sigma_0$.  
Then,
there exists a time $\tau>0$  such that  the pressure function $g(\cdot,t)=\sqrt{u(\cdot,t)}$  is smooth up to the interface $x_{n+1}=0$ and satisfies condition~($\star$) for all $t\in [0,\tau)$.
In particular, the interface $\Gamma_t$ between the flat side and
the strictly convex side is a smooth hypersurface for all $t$ in $0<t\leq \tau$ and it
moves by the \em{$*_{k-1}$} flow.
\end{thm}

\begin{rem} In the case of a two-dimensional surface in $\R^3$ the flow \eqref{eqn-qk} becomes
the well studied harmonic mean curvature flow. In this case, Theorem \ref{thm-main2}  was 
established in \cite{CD}.  
Following the result in \cite{CD} one may consider a pressure function $g=u^p$,  for any number
$p \in (0,1)$,  and prove the short time existence of a solution to the \eqref{eqn-qk} flow which is
of class $C^{m,\gamma}$ (with $m, \gamma$ depending on $p$) so that the pressure function
$g$ is still smooth up to the interface and the interface moves by the $Q_{k-1}$ flow. The fact that the solution $\Sigma_t$ remains in the class
$C^{m,\gamma}$, for $t >0$, distinguishes this flow from other,
previously studied, degenerate free-boundary problems (such as the
Gauss curvature flow with flat sides, the porous medium equation and
the evolution p-laplacian equation) in which the regularity of the solution 
for $t >0$ does not depend on the regularity of the initial data.

\end{rem}

\medskip

The paper is organized as follows:  in section
\ref{sect-convex} we will present  some apriori estimates for strictly
convex surfaces, most of which  have been shown   by Andrews in \cite{An1}. In
section \ref{sect-existence} we will  prove   the existence of a solution to \eqref{eqn-qk}, as stated in 
Theorem \ref{thm-main1}. The uniqueness of solutions will be shown in section 
\ref{sect-uniqueness}. 
Section \ref{sect-flat} will be devoted to the proof of Theorem \ref{thm-main2}.

\section{Part I: The existence and uniqueness of a $C^{1,1}$ solution}

\subsection{Notation and evolution equations}

Since the right hand side of \eqref{eqn-qk} can be viewed as a
function of the second fundamental form matrix $A$, a direct
computation shows that its linearization is given by
$$\mathcal{L}(u) = \frac{\partial Q_k}{\partial h_l^i}\, 
\nabla_i \, \nabla_l u = a^{il}\, \nabla_i \, \nabla_l u$$
with 
\begin{equation}
\label{eq-coeff}
a^{il} =  \frac{\partial Q_k}{\partial h_l^i}.
\end{equation}   
Notice that if we compute $a^{il}$ in geodesic coordinates around
the point (at which the matrix $A$ is diagonal) we get
$$a^{ii} = \frac{\partial Q_k}{\partial\lambda_i}$$
and other elements being zero.

Also,  we have the evolution equations of the induced metric $g_{ij}$,
\begin{equation}\label{eqn-gij}
\frac{\partial g_{ij}}{\partial t}  = -2Q_k \, h_{ij}
\end{equation}
of the mean curvature $H$,
\begin{equation}\label{eqn-H}
\frac{\partial H}{\partial t}= \mathcal{L} H + \frac{\partial^2Q_k}{\partial h_q^p\partial h_m^l}\nabla^ih_q^p\nabla_ih_m^l + \frac{\partial Q_k}{\partial h_m^l}h_p^lh_m^p\,  H
\end{equation}
and of the speed $Q_k$,
\begin{equation}\label{eqn-eqk}
\frac{\partial Q_k}{\partial t} = \mathcal{L} Q_k + \frac{\partial}{\partial h_j^i}Q_kh^{il}h_{lj}Q_k.
\end{equation}

\subsection{Apriori estimates for strictly convex hypersurfaces}
\label{sect-convex}

The main result in this section relies on the work  of Andrews in \cite{An1}.
Before we state it, lets prove the following lemma.

\begin{lem}
\label{lem-lower}
Let $\Sigma$ be  a convex hypersurface (not necessarily strictly convex) such that
it contains a ball $B_{\rho}(0)$, centered at the origin, of radius $\rho$. Then there
exists a constant $\delta = \delta(\rho) > 0$, 
depending  only on $\rho$ and the diameter of $\Sigma$, such that   that for every $q\in \Sigma$, we have 
$$\langle q, \nu\rangle \ge \delta$$
where $\nu$ denotes the outer unit normal to $\Sigma$ at $q$. 
\end{lem}

\begin{proof}
Since our hypersurface is convex and the origin is enclosed by it, we have that
$\langle q,\nu\rangle \ge 0$ for every $q\in \Sigma$.
Assume that the conclusion of the Lemma is not true. 
Then we can always find a point $q\in \Sigma$ for which
$\langle q,\nu\rangle$ can be made arbitrarily small.
Observe that if $\alpha$ is an angle between $q$ (considered as a position vector) and $\nu$, 
$$\langle q,\nu\rangle = |q|\, \cos \alpha \ge \rho\, \cos\alpha$$
from where it follows that $\alpha$ can be made as close to $\frac{\pi}{2}$ as
we want by our choice of $q$. Let $\tilde{q}$ be an  intersection point of $\Sigma$ with the  line that 
contains the origin and has the normal vector at the point $q$ as its direction. 
Since this line intersect  $\Sigma$ at two points,  we take $\tilde q$ to 
be the closer intersection point to $q$. Consider the  triangle with
vertices being the origin, $q$ and $\tilde{q}$. Since  $\alpha \approx  \frac{\pi}{2}$, 
since   the angle between the normal vector $\nu$ and the tangent vector $\tau$ at the point $q$ is 
$\frac{\pi}{2}$ and because  our surface lies completely on one side of $\tau$,  it follows 
that the angle of this triangle at the point $q$ (denote it by $\angle q$) is close to zero.
The angle at the origin $\angle 0 \approx \frac{\pi}{2}$ and therefore
$$|\tilde{q}| \approx  |q\tilde{q}| \, \cos\angle\tilde{q}$$
which implies
$$\cos \angle\tilde{q} \approx  \frac{|\tilde{q}|}{|q\tilde{q}|} \ge \frac{\rho}{C}$$
since the diameter $\diam(\Sigma_0) \le C$. Hence 
$$\angle\tilde{q} \le \frac{\pi}{2} - \eta$$
for a uniform constant $\eta$, depending only on $\rho$ and the diameter
of $\Sigma$. This shows that  $\angle q \ge \frac{\eta}{2}$ which
contradicts our choice of $q$.
\end{proof}

\begin{thm}
\label{thm-strict-conv}
Let $\Sigma_0$ be a strictly convex hypersurface in
$\mathbb{R}^{n+1}$ and let $\Sigma_t$ be a family  of hypersurfaces
evolving by (\ref{eqn-qk}). Let $\tau > 0$ be such that 
a ball $B_{\rho}$, of radius $\rho$ is contained in $\Sigma_{\tau}$. Then,  there
exists a positive constant $C  = C(\rho, \|\Sigma_0\|_{C^{1,1}})$ so that
$$|A(t)| \le C, \qquad  \forall \,\, 0 \le t \le \tau,$$
where $A(t)$ is the second fundamental form of $\Sigma_t$.
\end{thm}

\begin{proof}
Assume with no loss of generality that the origin is enclosed by $\Sigma_0$ and 
$$B_\rho(0)\,\,\,  \mbox{is contained in } \,\, \Sigma_{\tau}.$$
As in \cite{DS} we consider the quantity
$$\mathcal{F} = \langle F,\nu\rangle + 2\, t\, Q_k.$$
Since the speed $Q_k$ is a homogeneous function of the principal curvatures of degree one, a similar computation as in \cite{DS} yields to 
\begin{equation}
\label{eq-dot-pr}
\frac{\partial}{\partial t}\langle F,\nu\rangle = \mathcal{L}\langle F,\nu\rangle - 2Q_k + \frac{\partial Q_k}{\partial h_j^i}h^{il}h_{lj}\, \langle F,\nu\rangle.
\end{equation}
Hence
\begin{equation}
\label{eq-mon}
\frac{\partial}{\partial t}\mathcal{F} = \mathcal{L}\mathcal{F} +  \frac{\partial Q_k}{\partial h_j^i}h^{il}h_{lj}\, \mathcal{F}.
\end{equation}
As in \cite{D} we have
\begin{equation}
\label{eq-positive}
\sum_{i=1}^n \frac{\partial Q_k(\lambda)}{\partial \lambda_i}\lambda_i^2 \ge \frac{k}{n-k+1}\, Q_k^2(\lambda) \ge 0.
\end{equation}
This together with the maximum principle applied to (\ref{eq-mon}) imply
$$\mathcal{F}_{\min}(t) \nearrow \,\,\, \mbox{in} \,\,\, t,$$
and therefore by Lemma \ref{lem-lower},
$$\mathcal{F}_{\min}(t) \ge \mathcal{F}_{\min}(0) = \langle F,\nu\rangle_{\min}(0) \ge \delta > 0, \qquad  \mbox{for all} \,\,\, t \ge 0.$$
The speed $Q_k$ is concave and from the evolution equation for $\frac{H}{\mathcal{F}}$ one can easily deduce,
$$\frac{\partial}{\partial t} ( \frac{H}{\mathcal{F}}   )  \le \mathcal{L}(\frac{H}{\mathcal{F}}) +
\frac{2}{\mathcal{F}}\, \frac{\partial Q_k}{\partial h_j^i}\,  \nabla_j\mathcal{F}\, \nabla^i (\frac{H}{\mathcal{F}} ).$$
By the maximum principle applied to the previous inequality we get
$$\sup_{\Sigma_t}\frac{H}{\mathcal{F}} \searrow  \,\,\, \mbox{in} \,\,\, t.$$
Since $|\langle F,\nu\rangle| \le |F| \le C$, uniformly in $t$, there are uniform constants $C_1, C_2$ so that
\begin{equation}
\label{eq-H}
H(\cdot,t) \le C_1 + C_2\, Q_k.
\end{equation}
Also, because  our surface is convex, the second fundamental form is controlled by the mean curvature $H$. In addition, since  we assume that the  ball $B_{\rho} \subset \Sigma_{\tau}$, by Lemma \ref{lem-lower} there exists $\delta = \delta(\rho)$ so that 
$$\langle F, \nu\rangle \ge 2\delta, \qquad \mbox{for all} \,\,\, t\in [0,\tau].$$
As in \cite{An1} we consider the quantity $$\frac{Q_k}{\langle F,\nu\rangle - \delta}.$$
The evolution equation for $Q_k$, (\ref{eq-dot-pr}) and (\ref{eq-positive}) yield to:
\begin{eqnarray}
\label{eq-long}
\frac{\partial}{\partial t} \left (\frac{Q_k}{\langle F,\nu\rangle - \delta} \right ) &=& \mathcal{L}
\left (\frac{Q_k}{\langle F,\nu\rangle - \delta} \right ) + \frac{2}{\langle F,\nu\rangle - \delta}\, \frac{\partial Q_k}{\partial h_j^i}\, \nabla^i \,  \left (\frac{Q_k}{\langle F,\nu\rangle-\delta} \right )\, \nabla_j\langle F,\nu\rangle \nonumber \\
&+& \frac{Q_k}{(\langle F,\nu\rangle-\delta)^2} \left (2Q_k - \delta \, \frac{\partial Q_k}{\partial h_j^i}h^{il}h_{lj} \right ) \nonumber \\
&\le&  \mathcal{L} \left (\frac{Q_k}{\langle F,\nu\rangle - \delta} \right ) + \frac{2}{\langle F,\nu\rangle - \delta}\, \frac{\partial Q_k}{\partial h_j^i}\, \nabla^i \left (\frac{Q_k}{\langle F,\nu\rangle-\delta} \right )\, \nabla_j\langle F,\nu\rangle \nonumber \\
&+& \frac{Q_k^2}{(\langle F,\nu\rangle - \delta)^2}\, \left (2 - \delta\, \frac{k}{n-k+1}Q_k
\right ).
\end{eqnarray}
If $2 - \delta\, \frac{k}{n-k+1}\, Q_k \ge 0$, this implies that $Q_k \le \frac{2(n-k+1)}{k\, \delta}$ and because of   (\ref{eq-H}) we are done. If $2 - \delta\, \frac{k}{n-k+1}\, Q_k < 0$, the maximum principle applied to (\ref{eq-long}) yields to 
$$\frac{d}{dt} \left (\frac{Q_k}{\langle F,\nu\rangle-\delta} \right )_{\max} \searrow  \,\,\, \mbox{in} \,\,\, t$$
which, in particular  implies that 
$$Q_k(t) \le C$$
for a   constant $C$ that depends on $\delta = \delta(\rho)$ and that  $C^{1,1}$ norm of $\Sigma_0$.
\end{proof}

\subsection{The existence of a viscosity solution to (\ref{eqn-qk})}\label{sect-existence}

We will next show the following proposition which constitutes the existence 
part of Theorem \ref{thm-main1}. 

\begin{prop}
\label{prop-existence}
Under the initial assumptions of Theorem \ref{thm-main1},    there exists a $C^{1,1}$ solution
$\Sigma_t$  to the
\eqref{eqn-qk} flow up to some finite time $T$ at which the enclosed volume shrinks to
zero. 
\end{prop}

\begin{proof}
We begin by approximating  our initial  convex hypersurface $\Sigma_0$  by a
family of smooth strictly convex surfaces $\Sigma^{\e}_0$ (for example by the
mean curvature flow). We may assume that all $\Sigma^{\e}_0$ are contained
in $\Omega$, approaching $\Sigma_0 = \partial\Omega_0$ in the Hausdorff
distance as $\e \to 0$ and that $\Sigma^{\e}_0$ increase in $\e$. Let $\Sigma^{\e}_t$,  $t\in [0,T_{\e})$,
be the unique strictly solution of \eqref{eqn-qk} flow with initial data  $\Sigma^{\e}_0$. 
Then, $\Sigma_t^\e$ increases in $\e$. In addition,     $T_\e$ denotes the time at which the surface $\Sigma^\e_t$ shrinks to a point, then  
$T_{\e}$ increases in $\e$. Moreover, an easy application of the
comparison principle shows that  $ T_{\e} \le T_R$, where
$T_R$ is  the extinction times of a   sphere $S_R$ 
 which can be  placed  outside of $\Sigma_0$.
Hence, the limit 
$$T:= \lim_{\e\to 0} T_{\e}$$
exists.  Let $\tau <
T$. Then there exist $\rho > 0$ and $\e_0$ so that $B_{\rho} \subset
\Sigma^\e(\tau)$, for all $\e \le \e_0$. Theorem
\ref{thm-strict-conv} implies the uniform bound 
$$\|\Sigma^{\e}_t\|_{C^{1,1}} \le C(\tau,\|\Sigma\|_{C^{1,1}}), \qquad 
\mbox{for all}\,\,\,  t\in [0,\tau].$$ Since each $\Sigma^{\e}$ is smooth, the previous estimate 
shows that
$$\|\Sigma^{\e}_t\|_{C^2} \le C(\tau,\|\Sigma\|_{C^{1,1}}), \,\,\, t\in [0,\tau].$$
Also, since  $\Sigma^{\e}_t$ is increasing in $\e$, the Arzela-Ascoli theorem
and the previous estimate imply there is a $C^{1,1}$ limit 
$$\Sigma_t
:= \lim_{\e\to 0} \Sigma^{\e}_t$$ and that the convergence is in the $C^{1,1}$
norm.

We claim that $\Sigma_t$ is a viscosity solution to (\ref{eqn-qk}) in
the sense of Definition \ref{def-visc}.  To see that, let $\Sigma'_0 \subset  \Omega$ where $\partial\Omega = \Sigma_0$.  Then  
$\Sigma'_0$ is enclosed by $\Sigma^{\e}_0$,  for $\e$ sufficiently small. By
the comparison principle $\Sigma'_t$ is enclosed by $\Sigma^{\e}_t$, 
for $t > 0$. Since $\Sigma^{\e}_t$ increases in $\e$ and converges to
$\Sigma_t$, we have that $\Sigma'_t$ is enclosed by $\Sigma_t$,  for
$t > 0$. The second condition in Definition \ref{def-visc} can be
checked similarly.

We finally observe that the enclosed volume of $\Sigma_t$ shrinks to zero, as
$t \to T$. Indeed, assume otherwise. Then, there exists a sphere  $S_{2\rho}(P)$,
for some $\rho >0$,  which is enclosed by $\Sigma_t$,  for all $t <  T$. 
Fix a $t_0< T$ to be chosen momentarily.  Since $\Sigma_\e^ {t_0} \nearrow \Sigma_{t_0}$, as $\e \to 0$, this means
that there exists an $\e=\e(t_0)$ such that $S_{\rho }(P)$ is enclosed  inside
the surface $\Sigma_\e^ {t_0}$. The comparison principle then shows that the
vanishing time $T_\e$ of $\Sigma_t^\e$ satisfies 
$$T_{\e} \geq t_0 + T_\rho$$
where $T_\rho$ is the time at which the 
sphere $S_{\rho}(P)$ evolving by our \eqref{eqn-qk} flow shrinks to a point. On the other hand, since 
$$T \geq T_{\e} \geq t_0 + T_\rho$$
we will reach a contradiction provided that $t_0$ is chosen sufficiently close to $T$,
depending only on $\rho$. This completes the proof of our proposition.

\end{proof}

\begin{prop}
\label{prop-point}
Under the initial assumptions of Theorem \ref{thm-main1},    there exists a $C^{1,1}$ solution
$\Sigma_t$  to the
\eqref{eqn-qk} flow up to some finite time $T$ at which the surface shrinks to a point.
\end{prop}

\subsection{The uniqueness of a viscosity solution to (\ref{eqn-qk})}\label{sect-uniqueness}

In Proposition  \ref{prop-existence} we have constructed a viscosity solution
to (\ref{eqn-qk}). The question that arises is whether that
solution is unique in the class of viscosity solutions defined by
Definition \ref{def-visc}. We will give a positive answer to this 
question in the following proposition.

\begin{prop}\label{prop-uniqueness}
If $F_1(\cdot,t), F_2(\cdot,t)$ are two viscosity solutions to
(\ref{eqn-qk}),  then $$F_1 \equiv F_2.$$
\end{prop}

\begin{proof}
Let $\Sigma_0$ be our initial surface and $\Sigma^{\e}_0$ 
the strictly convex approximating surfaces   considered in Proposition  \ref{prop-existence},
which are all enclosed by $\Sigma_0$. We may assume, without loss of generality, that  
all $\Sigma^{\e}_0$ contain a  ball $B_{\rho}(0)$ centered at the origin. 
Let us  consider another strictly convex surface 
$\tilde{\Sigma}^{\e\delta}_0$ defined as a dilation of $\Sigma^{\e}_0$, that is,
if $\Sigma^{\e}_0$ is defined by an immersion $F^{\e}_0$, then
$$\tilde{F}^{\e\delta}_0= (1+\delta)\, F^{\e}_0$$
defines the surface $\tilde{\Sigma}^{\e\delta}_0$. 

\begin{claim}
For every $\delta > 0$ there exists an $\e = \e(\delta)$ so that 
the  surface $\tilde{\Sigma}^{\e\delta}_0$ defined as above is encloses the surface $\Sigma_0$.
\end{claim}

\begin{proof}
Observe that
$$|\tilde{F_0}^{\e\delta}| - |F_0| = |F^{\e}_0| - |F_0| + \delta|F^{\e}_0|$$
and that   $\|F^{\e}_0| - |F_0\| < \alpha(\e)$, where $\alpha(\e) \to 0$ as 
$\e \to 0$. Since all  $\Sigma^{\e}_0$ contain the ball  $B_{\rho}(0)$, we have 
$$|\tilde{F}^{\e\delta}_0| - |F_0| \ge -\alpha(\e) + \delta\, \rho > 0, \qquad \mbox{if} \,\, 
\alpha(\e) < \delta\, \rho.$$ Hence, by choosing   $\e := \e(\delta)$ 
sufficiently small so that $\alpha(\e) < \delta\, \rho$, we guarantee that  
$\tilde{\Sigma}^{\e\delta}_0$  encloses  $\Sigma_0$. 
\end{proof}

From now on  fix $\delta >0$, set $\e=\e(\delta)$ and consider
the corresponding surfaces $\Sigma^{\e}$ and $\tilde \Sigma^{\e\delta}_0$. Then, 
$\tilde \Sigma^{\e\delta}_0$ encloses $\Sigma_0$ and $\Sigma_0$ encloses $\Sigma^{\e}$. We write briefly that
$$\Sigma^{\e}_0  \prec \Sigma_0  \prec  \tilde \Sigma^{\e\delta}_0.$$
By Definition \ref{def-visc}, we will have
\begin{equation}\label{eqn-encl}
\Sigma^{\e}_t  \prec \Sigma_t  \prec  \tilde \Sigma^{\e\delta}_t, \qquad \forall t < T_\e
\end{equation}
where $T_\e$ denotes the extinction time of $\Sigma^{\e}_t $, since $\Sigma_t$ is a viscosity solution.

Denote by $\Sigma_t^\e, \tilde \Sigma^{\e\delta}_t$ the smooth  solutions  of \eqref{eqn-qk} starting at  $\Sigma^{\e}_0, \tilde \Sigma^{\e\delta}_0$ respectively
(note that both solutions are unique because their initial data are strictly convex). 
It is easy to see that $\tilde{\Sigma}^{\e\delta}_t$ is given by the immersion 
$$\tilde{F}^{\e\delta}(\cdot,t) = (1+\delta) \, F^{\e}(\cdot,\frac{t}{(1+\delta)^2}).$$
By above $\tilde{F}^{\e\delta}(t\, (1+\delta)^2) = (1+\delta) \, F^{\e}(t)$,
and we compute
\begin{eqnarray*}
\frac{\partial}{\partial t} \, [\, \tilde{F}^{\e\delta}(t\, (1+\delta)^2) - F^{\e}(t)  \, ]  &=&
 Q_k^\e\, \nu^\e(t)\, (1+\delta)  - Q_k^\e\, \nu^\e(t) \\
&=&  \delta \,  Q_k^\e\, \nu^\e(t)
\end{eqnarray*}
which implies that 
\begin{equation}
\label{eq-diff}
\frac{\partial}{\partial t} \, | \, \tilde{F}^{\e\delta} (\cdot,t\, (1+\delta)^2) - F^\e(\cdot,t) \, | \le 
\delta\,  Q_k^\e(t)
\end{equation}
where $Q_k^\e(t)$ is the speed of the flow (\ref{eqn-qk}) starting at
the  surface $\Sigma^\e_0$. 

Denote, as  before by  $T_\e$ the extinction time of $\Sigma^\e_t$ and let 
$T  = \lim_{\e \to 0} T^\e$.  Take any $\tau < T$. Then, there  exist $\delta_0 > 0$ and $\rho = \rho(\tau)$ so that 
$$B_{\rho} \subset \Int(\Sigma^{\e}_t), \qquad \forall \delta < \delta_0, \qquad \forall
t  \le \tau < T^\e \le T$$
for an  $\e=\e(\delta)$ which is defined as above. 
By Theorem \ref{thm-strict-conv}  it follows that the speed $\tilde Q_k^{\e\delta}$ of
$\tilde \Sigma_t^{\e\delta}$ satisfies 
\begin{equation}
\label{eq-uni-est}
|\tilde Q_k^{\e\delta} (t)| \le C(\tau, \|\Sigma\|_{C^{1,1}}), \qquad \forall \,\, \delta \le \delta_0, \quad \forall \,\, t \le \tau.
\end{equation} 
Estimate (\ref{eq-diff}) yields to 
\begin{eqnarray}
\label{eq-one-part}
| \tilde{F}^{\e \delta} (t\, (1+\delta)^2) - F^\e(t)| &\le& |\tilde{F}^{\e\delta}_0 - F^\e_0| + \delta\,  C(\tau)  \\
&\le& \delta\, ( |F^\e_0| + C(\tau)) \nonumber  \\
&\le& \tilde C(\tau)\, \delta \nonumber 
\end{eqnarray}  
Since $\tilde{F}^{\e\delta} (t)$ solves the \eqref{eqn-qk} with   and $\tilde{Q}_k^{\e\delta}(t) = \frac{Q_k^\e(t \,  (1+\delta)^{-2})}{1+\delta}$,  it follows from (\ref{eq-uni-est})  that
\begin{eqnarray}
\label{eq-two-part}
|\tilde{F}^{\e\delta} (t(1+\delta)^2) - \tilde{F}^{\e\delta} (t)| &\le& C(\tau)\,  \tau \, ((1+\delta)^2 - 1) + |\tilde{F}^{\e\delta}_0 - F^\e_0|  \\
&\le& ( \tilde{C}(\tau)  + |F^\e_0| ) \, \delta \leq \bar C(\tau) \, \delta  \nonumber 
\end{eqnarray}
Combining (\ref{eq-one-part}) and (\ref{eq-two-part}) yields to 
\begin{equation}
\label{eq-close}
|\tilde F^{\e\delta} (t) - {F}^\e(t)| \le C(\tau)\,  \delta, \qquad \forall t  \in [0,\tau], \,\,\, \delta < \delta_0
\end{equation}
for $\e=\e(\delta)$ defined as above. 
Here,  $C(\tau)$ is a uniform constant that does not depend on $\delta$. In particular, (\ref{eq-close}) implies that the viscosity solutions which are obtained as  the limits of $\{F^\e(t)\}$ and $\{\tilde{F}^{\e\delta}(t)\}$,  as $\delta \to 0$ (whose existence is justified by Proposition \ref{prop-existence}) are the same.

We will now conclude the proof of uniqueness. Let $\Sigma^1_t$ and $\Sigma^2_t$ be two families of viscosity solutions to (\ref{eqn-qk}) starting at $\Sigma_0$.  
Let $\{\Sigma^{\e}\}$ and $\{\tilde{\Sigma}^{\e\delta}\}$ be two families of approximations of $\Sigma$ by strictly convex surfaces taken as above. It follows by the  Definition \ref{def-visc}
of viscosity solutions, that both $\Sigma^1_t$ and $\Sigma^2_t$
satisfy \eqref{eqn-encl}. 
Hence,  from  (\ref{eq-close}) we have 
$$|F^1(t) - F^2(t)| \le |\tilde F^{\e\delta}(t) - {F}^{\e}(t)| \le C(\tau)\, \delta,
\qquad \forall t < \tau.$$  Letting   $\delta\to 0$ we  obtain that $\Sigma^1_t = \Sigma^2_t$ for all $t\in [0,\tau]$. Since $\tau < T$ was arbitrary, we conclude that $\Sigma^1_t = \Sigma^2_t$,  for all $t\in [0,T)$, which finishes  our proof. 
\end{proof}

\subsection{Discussion on strict convexity}  In this section we will give
some observations as to  when a  convex surface,  which is not necessarily strictly convex, 
actually becomes strictly convex 
as soon as it moves away from the initial surface.  We first observe that the speed $Q_k$
is bounded from below away from zero at points of $\Sigma_t$ which are away
from the initial surface $\Sigma_0$.

\begin{prop}\label{prop-convex} Assume that $\Sigma_0$ is a $C^{1,1}$ compact convex hypersurface embedded in $\R^{n+1}$ 
and  let $\Sigma_t$   denote the unique 
$C^{1,1}$   solution
of the \eqref{eqn-qk} flow with initial data $\Sigma_0$ which exists on $0 < t < T$. Assume that at some point $P   \in \Sigma_{t_0}$,  with   $t_0 < T$,  we have  $d_P:=\dist (P, \Sigma_0) >0$. Then, there exist positive constants   $\delta > 0$, $\tau > 0$ and $c >0$,  depending only on $d_P$ and the diameter of the initial surface $\Sigma_0$, such that the speed $Q_k \geq c >0$ at all points $Q \in \Sigma_t$  with  $\dist_{M^n} (P,Q) < \delta$ and $t\in [t_0-\tau,t_0+\tau]$.
\end{prop}

\begin{proof} The idea of the proof is simple. For the given point $P$ we will consider the
quantity 
$$\F := \langle F - P,  \, \nu \rangle + 2t\, Q_k$$
which we already introduced  in the proof of Theorem \ref{thm-strict-conv} where  we 
also showed
that its minimum is increasing in time. Since $B_{{d_P}/{2}}(P)$ is strictly contained
in the initial surface $\Sigma_0$, Lemma \ref{lem-lower} implies that at time $t=0$ we
have $\F_{\min}(0) \geq c_0 >0$
for some constant $c_0 >0$ depending only on $d_P$ and the diameter of  $\Sigma_0$. Hence, 
$$\F(p,t_0) = 2t_0\, Q_k(p,t_0)  \geq c_0 >0$$
implying that the speed $Q_k$ of the surface is strictly positive near $P$
(here we denote by $p \in M^n$ the point such that $F(p,t_0)=P$). 

To make the above argument rigorous, we  let $\Sigma_0^i$ be a decreasing sequence of strictly convex surfaces 
which approximate $\Sigma_0$, i.e. we have $\Sigma^i_0 \searrow  \Sigma_0$
in the $C^{1,1}$ norm. We denote by $\Sigma^i_t$ the solution of the \eqref{eqn-qk}
flow with initial data $\Sigma^i_0$ and by $T^i$ its vanishing time. Then, it follows
from the proof of Theorem \ref{thm-main1} that $\Sigma^i_t \searrow  \Sigma_t$
in the $C^{1,1}$ norm for all $0 < t < T$. Pick points $P^i \in \Sigma^i_{t_0}$ such
that $P^i \to P$ and choose $i_0$ sufficiently large so that 
$$\dist(P_i, \Sigma^i_0) \geq \frac{d_p}{2}:=d >0, \qquad \forall {i \geq i_0}.$$
Fixing $i \geq i_0$  for the moment, we consider the quantity
$$\F^i (\cdot,t) =  \langle F^i  - P,  \, \nu \rangle + 2t\, Q_k^i$$
with $F^i$ and $Q_k^i$ denoting the position function and speed of $\Sit$
respectively. Since, $B_{d}(P_i)$ is contained in the initial surface $\Sii_0$,
Lemma \ref{lem-lower} implies that
$$\F^i_{\min} (0) \geq c_0>0$$
for some constant $c_0$ depending only on $d$ and the diameter of $\Sigma_0$. The maximum
principle applied to the evolution of $\F^i$ (as in the proof of Theorem \ref{thm-strict-conv}) implies that
\begin{equation}\label{eqn-verygood}
\F^i_{\min} (t) \geq c_0 >0, \qquad \forall t < T.
\end{equation}
Denote by $p \in M^n$ the point at which $F(p,t_0)=P$. 
For the given time $t_0$, choose $\tau >0, \e >0$ sufficiently  small  so  that
$$|F^i(q,t) - P | \leq \frac {c_0}2, \qquad   \forall  t \in [t_0-\tau, t_0+\tau], \quad \dist_{M^n}(p,q) < \e$$
with $c_0$ the constant in \eqref{eqn-verygood}. 
The a priori estimates in Theorem \ref{thm-strict-conv} imply that $\tau$ and $\e$
can be chosen to be independent of $i$. It follows from \eqref{eqn-verygood}
that
$$
Q_k^i(q,t) \geq \frac {c_0}{4t} \geq c >0, \qquad   \forall  t \in [t_0-\tau, t_0+\tau], \quad \dist_{M^n}(p,q) < \e.
$$
The proposition now follows from the observation that $\Sit$ have  uniformly bounded
$C^{1,1}$ norms after we pass to the limit. Passing to the limit yields  
\begin{equation}\label{eqn-verygood2}
Q_k (q,t) \geq  c  >0, \qquad   \forall  t \in [t_0-\tau, t_0+\tau], \quad \dist_{M^n}(p,q) < \e.
\end{equation}

\end{proof}

\begin{cor}  Assume that $\Sigma_0$ is a $C^{1,1}$ compact convex hypersurface embedded in $\R^{n+1}$ 
and  let $\Sigma_t$   denote the unique 
$C^{1,1}$   solution
of the \eqref{eqn-qk} flow with initial data $\Sigma_0$ which exists on $0 < t < T$. If $\dist(\Sigma_t,\Sigma_0) >0$,  then $\Sigma_t$ is strictly convex. 
\end{cor}

\begin{proof} We combine Proposition \ref{prop-convex} and a recent constant rank theorem  by Bian and Guan in \cite{BG}. By the previous proposition $Q_k \geq C(t) >0$ on $\Sigma_t$. 
This in particular implies that the equation \eqref{eqn-qk} is strictly
parabolic and that the surface $\Sigma_t$ is smooth. It then follows from the constant
rank theorem in  \cite{BG} that at  any given time $t >0$ the rank of the second fundamental form
of the surface $\Sigma_t$ is constant. Since  $\Sigma_t$ is a smooth compact 
surface, there exists at least a point $P$ at which $\Sigma_t$ is strictly convex,
which forces the whole surface to be strictly convex, finishing the proof. 

\end{proof}

We next observe that combining the previous corollary and the main result of Dieter in \cite{D} we obtain the following:

\begin{cor} Assume that $\Sigma_0$ is a $C^{1,1}$ compact convex hypersurface embedded in $\R^{n+1}$ 
and  let $\Sigma_t$   denote the unique 
$C^{1,1}$   solution
of the \eqref{eqn-qk}  flow with initial data $\Sigma_0$ which exists on $0 < t < T$. If $S_{k-1}^n >0$
uniformly on $\Sigma_0$, then the solution $\Sigma_t$ is strictly convex for all $t >0$.
\end{cor}

\begin{proof} The main result by Dieter in \cite{D} shows that if $S_{k-1}^n>0$
uniformly on $\Sigma_0$, then the speed $Q_k$ satisfies the bound $Q_k > C(t) >0$ 
on $\Sigma_t$. Hence, as in the previous corollary $\Sigma_t$ is strictly convex. 

\end{proof}

One may ask:  does it follow  from Proposition  \ref{prop-convex} that the surface
$\Sigma_t$  is strictly convex locally near points $P$ which are  away from the initial surface ? 
The same proposition 
shows  that the $k$ largest principal curvatures $\lambda_1, \cdots,\lambda_k$ are positive at points of $\Sigma_t$ 
which are away from $\Sigma_0$. However, some of the other principal curvatures
may vanish. On the  other hand, since the constant rank
theorem in \cite{BG} is local, Proposition \ref{prop-convex} implies  that the rank of the second
fundamental form of  the surface $\Sigma_t$
is constant on each   connected component of $\Sigma_t \setminus \Sigma_0$. Hence, 
those connected components that  contain at least one point at which the surface is strictly convex 
are indeed strictly convex. The question as to whether $\Sigma_t \setminus \Sigma_0$
is always strictly convex remains open for investigation. 
However, our discussion above leads to  the following observation.

\begin{cor}  Assume that $\Sigma_0$ is a $C^{1,1}$ compact convex hypersurface embedded in $\R^{n+1}$ 
and  let $\Sigma_t$   denote the unique 
$C^{1,1}$   solution
of the ($*_n$) flow with initial data $\Sigma_0$ which exists on $0 < t < T$. Assume that at some point $P   \in \Sigma_{t_0}$,  with   $t_0 < T$,  we have  $d_P:=\dist (P, \Sigma_0) >0$. Then, there exist positive constant  $\delta >0$,  depending only on $d_P$ and the initial surface $\Sigma_0$, such that $\Sigma_t$ is strictly convex  at all points $Q \in \Sigma_t$  such that $\dist (P,Q) < \delta$.
\end{cor}

\begin{proof} Denote by $\lambda_j$,
$j=1,\cdots, n$ the principal curvatures of the surface $\Sigma_t$ and
assume that 
$
\lambda_1 \geq \lambda_2  \geq \cdots \geq \lambda_n.
$
Hence, 
$$Q_n^i \approx  \frac{ \lambda_1^i \cdots \lambda_{n-1} \, \lambda_n}{ \lambda_1 \cdots \lambda_{n-1} } = \lambda_n$$
in the sense that the ratio of $Q_n /\lambda_n$ is bounded from above and
below by positive constants depending only on the dimension $n$.
It follows  from \eqref{eqn-verygood2}  that
$$\lambda_n(q,t) \geq c >0, \qquad   \forall  t \in [t_0-\tau, t_0+\tau], \quad \dist_{M^n}(p,q) < \e
$$
where $c >0$ depends only on $d_p$, $\Sigma_0$ and  the dimension $n$. Since, 
$\lambda_n$ is the smallest of the principal curvatures of $\Sigma_t$, our
result readily follows.

\end{proof}

We conclude our  discussion with the open questions: (i) Does the surface $\Sigma_t$ become strictly convex before it  becomes extinct  ? Or in the case that (i) is not valid:  (ii)   Does the surface $\Sigma_t$ always shrink to a point ?


\section{Part II: The evolution of a surface with flat sides}\label{sect-flat}

In this section we will study the evolution of a convex surface with flat sides under the
\eqref{eqn-qk} flow for $2\leq k\leq n$. Our goal is to give the proof of Theorem \ref{thm-main2}.
Let us briefly outline the steps of its proof. 
The \eqref{eqn-qk}  flow can be seen as a free boundary problem arising from the
degeneracy near the flat side of the fully nonlinear parabolic PDE
which describes the flow. Via  a coordinate change we will   show that solving this free boundary problem is equivalent to solving  an {\it initial value problem} of the form
\begin{equation}\label{eqn-ivp}
\begin{cases}
 Mw =0
& \qquad \mbox{on}\,\,{\mathcal{D}}\times [0,\tau]\\\
w=w_0 & \qquad \mbox{at} \,\,\,\,\,\,\,t=0%
\end{cases}
\tag{IVP}
\end{equation}
on the unit ball in $\R^n$
$$ \mathcal{D} = \{ x:=(x_1,\cdots,x_n) \in \R^n ; \,\,\, |x| \leq 1 \}.$$

\vskip 0.1 in
\noindent The operator $M$, defined as
$$Mw= w_t - F(t,x_1,\cdots,x_n, w,Dw, D^2w)$$
 is a fully non-linear operator which becomes degenerate
at $\partial \mathcal D$, the boundary of  $\mathcal D$. We will apply the inverse function theorem between appropriately defined Banach spaces to show that this problem admits
a solution on $0 \leq  t \leq \tau$, for some $\tau >0$..

The linearization of the operator $M$ at a point $\bar w$ close to the initial data $w_0$
can be modeled  (after straightening the boundary near  $z:=x_1=0$)   on the degenerate equation
\begin{equation}\label{eqn-hl}
f_t =z^2\,  \tilde a_{11}  f_{11} + 2 \, z\, \tilde a_{1i}  f_{1i} + \tilde a_{ij} f_{ij} +z \tilde b_{1} f_{1} + \tilde b_i \,  f_{i} + \tilde c \, f,\qquad 
 i,j \neq 1
\end{equation}  
\noindent  on the half space $z >0$ with no extra conditions
on $f$ along the boundary $z=0$.     
The diffusion in the above equation is
governed by the {\em singular}  Riemannian metric $$ds^2=d \bar s^2+|dt|$$ where
\begin{equation}\label{eqn-metric}
d \bar s^2 = {dz^2 \over z^2}+dx_2^2+\cdots + dx_n^2.
\end{equation}
\par\noindent We notice that the distance (with respect to the singular metric
$\bar s$) of an interior point ($z>0$) from the boundary ($z=0$) is {\em
infinite}. This distinguishes our problem from other, previously
studied, degenerate free-boundary problems such as the degenerate
Gauss curvature flow \cite{DH2}, \cite{DL} and the porous medium equation
\cite{DH1}.
\vskip 0.01in
The results in this part are generalizations,  in dimensions $n \geq 3$,  of the
results in \cite{CD} for the harmonic mean curvature flow in dimension $n=2$.
Their proofs are similar to the corresponding proofs in \cite{CD}. We will only
give the main steps,  referring the reader to \cite{CD} for
the details.

\subsection {Local Change of Coordinates}\label{local}
In Section \ref{global} we will give the global change of coordinate which transforms our free-boundary
problem \eqref{eqn-qk} to a degenerate problem of the form \eqref{eqn-ivp} on  a domain with fixed boundary. Since  the 
computations there are quite involved, to motivate our discussion we will present here a local
change of coordinates near the interface which fixes the free boundary and we will
give the definitions of the appropriate Banach spaces where the existence of solutions will be 
shown.

We will assume throughout this section   that the surface $\Sigma_0$ belongs to the class $\mathfrak S$, as defined in the introduction. Let  $\Sigma_t$ be  a solution to the $\eqref{eqn-qk}$ flow on
$[0,\tau)$, for some $\tau>0$ such that  $\Sigma_t= \Sigma_t^1\cup \Sigma_t^2$, with $\Sigma^1_t$ flat and $\Sigma^2_t$ strictly convex. Let $P_0(x_1,\cdots, x_n, 0)$ be a point on the interface 
$\Gamma_{t_0}$,  for $t_0>0$ sufficiently small. Then, the strictly
convex part of surface $\Sigma^2_t$, $t<t_0$ can be expressed locally
around $P_0$ as the graph of a function $z=u(x_1,\cdots x_n, t)$. Let $g=\sqrt{u}$
 be the pressure function. Assuming that  $g$ is of class $C^2$  up to the interface and satisfies condition~($\star$), then we solve locally around the point $P_0$ the equation $z=u(x_1,\cdots x_n, t)$ with respect to $x_1$. This yields to the map
$x_1=f(z,x_2,\cdots x_n,t)$. 
 The condition ~($\star$) on $g$ expressed in terms of $f$ gives the following {\em non-degeneracy condition~($\star \star$)}  in a small neighborhood of $z=0$:
\begin{equation}
 \left (
\begin{split}
-z^{\frac{3}{2}}\,f_{11}\,\,\, & z^{\frac{3}{4}}\,f_{12}\,\,\, &  &\,\,\,  z^{\frac{3}{4}}\,f_{12}\\
z^{\frac{3}{4}}\,f_{12}\,\,\, &\,\,\,\,\, & \,\,\,\,\,&\,\,\,  \,\,\,\,\\
 \,\, & \,\,\, & { - F_{ij}} &\,\,\,\,\,\,\,\,\, \\
z^{\frac{3}{4}}\,f_{1n}\,\,\, & \,\, \,\,\,\,\, & \,\,\,\,&\,\,\,  \,\,\,\,
\end{split}\right ) \geq \bar \lambda I  \tag {$\star \star$}
\end{equation}
in the sense that the eigenvalues of  the above matrix are bounded from below by a number $\bar \lambda >0$. 
\noindent $\{F_{ij}\}$ is the Hessian matrix for the function $f$ with respect to the tangential directions.

\noindent  When $\Sigma_t$ evolves by the \eqref{eqn-qk} flow, then the function
$f$ evolves by a fully-nonlinear evolution equation of the from 
\begin{equation}\label{eqn-ff}
f_t = -\frac{S_k^n([b_{ij}])}{S_{k-1}^n([b_{ij}])}
\end{equation}
where $b_{ij}$ can be expressed in terms of $f$ and
its first and second derivatives. The details are given in Section \ref{sec-th2}. 
Equation \eqref{eqn-ff} becomes degenerate
near the interface $z=0$. Its linearization near the interface $z=0$ is of the form \eqref{eqn-hl}. 
Our goal is to construct a smooth solution to this equation by using the inverse function theorem between appropriately defined Banach spaces which are
scaled according to the singular metric \eqref{eqn-metric}. The main step is
to show  existence in an appropriately weighted $C^{2+\alpha}_{w,\bar s}$ space,
with respect to the metric $\bar s$. The definition of these spaces will be given in
the next section. Once  a $C^{2+\alpha}_{w,\bar s}$ solution is established,  one  can 
prove  the existence of a $C^\infty$ solution by repeated differentiation. Since this 
is similar to the results in \cite{CD} we will omit its proof. 

\subsection{The H\"older  spaces with respect to the singular metric }\label{def:ban}

We will define in this section the Banach spaces on which solutions of degenerate equations of the form \eqref{eqn-hl} are naturally defined. 

We will denote for the next of this section by $\bar x$ points $\bar x:=(x_2,\cdots,x_n) \in \R^{n-1}$
and we will consider points $(z,\bar x) \in \R^{n}$. 
\noindent Let
$\mathcal{A}$ be a compact subset of the half space $\{\,(z,\bar x)\in \R^{n}: \,\,
z\ge 0\,\}$ such that $0 \in {\mathcal
A}$. We define:
$$\begin{array}{lcl} \mathcal{A}^{\circ} & := &
\{\,\,\bar x \,\in\,\R^{n-1}\,:(0,\bar x)\,\in\, \mathcal{A}\,\}
\\  
Q_\tau  &: = &
{\mathcal A}\times [0,\tau], \,\,\,\,  \tau>0\\
 Q_\tau^{\circ}&:= &{\mathcal
A}^{\circ}\times [0,\tau], \,\, \tau >0.
\end{array}
$$

\noindent We define the 
{\it hyperbolic distance} $\bar s(P_1,P_2)$ between two points 
$P_1=(z_1,\bar x_1)$ and $P_2=( z_2,\bar x_2) $ in $\mathcal{A}$ with $z_1>0$, $ z_2>0$ to  be:

\noindent
$$ \bar s(P_1,P_2):= \sqrt{| \ln z_1- \ln z_2|^2+ |\bar x_1- \bar x_2|^2} ,\,\,\,\hbox{  if}\,\,\,  0 < z_1, z_2  \leq 1 $$
\noindent otherwise it is defined to be equivalent to the standard euclidean metric.

\smallskip

\noindent We define the 
{\it parabolic hyperbolic distance} between two points $\tilde P_1=(z_1,\bar x_1,t_1)$ and $\tilde P_2=(z_2,\bar x_2,t_2)$  with  $z_1>0$, $z_2>0$ to be:
$$s(\tilde P_1,\tilde P_2):= \bar s(P_1,P_2)+\sqrt{|t_1-t_2|}$$ 
\noindent where $P_1=(z_1,\bar x_1),\,P_2=(z_2,\bar x_2)$.

\smallskip

We will next define H\"older function  spaces with respect to the metric $s$. 

\noindent Given a function $f$ on
$\mathcal{A}$ we define:
\begin{equation*}
f^{\circ}(\bar x ) : = f(0,\bar x),\qquad 
\tilde{f}(z,\bar x) : =  \frac 1{\sqrt{z}} \,(\,f(z,\bar x)-f^{\circ}(\bar x)\,)
\end{equation*}

\noindent Analogously, given a function $f$ on $Q_\tau$ we define:
\begin{equation*}
f^{\circ}(\bar x,t ) : = f(0,\bar x,t), \qquad 
\tilde{f}(z,\bar x,t) : =  \frac 1{\sqrt{z}} \,(\,f(z,\bar x,t)-f^{\circ}(\bar x,t)\,)
\end{equation*}

\vskip 0.1 in
Let $0<\alpha\leq 1$. We will  define the weighted  H\"older space $C^{\alpha}_{w,\bar s}(\mathcal{A}$) in terms of the above distance.
\noindent We start defining the H\"older semi-norm:
$$\|\,f\,\|_{H^\alpha_{\bar s}(\mathcal{A})}: = \sup\limits_{P_1\not= P_2\in \tilde \A } \frac{|\,f(P_1) -
f(P_2)\,|}{ \bar s [P_1, P_2]^{\alpha}}\,$$
and the norm
$$\|\,f\,\|_{C^\alpha_{\bar s}(\mathcal{A})}: = \|\,f\,\|_{C^0 (\mathcal{A})}
+ \|\,f\,\|_{H^\alpha_{\bar s}(\mathcal{A})}$$ 
\noindent where $\|\,f\,\|_{C^{0}({\mathcal{A}})}:=\sup\limits_{P\in\mathcal{A}}|\,f(P)\,|$.

\noindent  Moreover, we define: $$\|\,f\,\|_{C^{0}_w(\mathcal
A)}:=\|\,f^{\circ}\,\|_{C^{0}(\mathcal A^{\circ})}+\|\,\tilde
f\,\|_{C^{0}({\mathcal A})}.$$

\begin{defn}[The space $C^{\alpha}_{w,\bar s}(\mathcal{A})$]
A function $f$ belongs to $C^{\alpha}_{w,\bar s}(\mathcal{A})$ iff
$f^{\circ}\in C^{\alpha}(\mathcal{A}^{\circ}) \,\, \text{and}\,\,
\,\tilde{f}\in C^{\alpha}_{\bar s}({\mathcal{A}})$. 
The norm of $f$ in the
space $C^{\alpha}_{w,\bar s}(\mathcal{A})$ is defined as:
$$\|\,f\,\|_{C^{\alpha}_{w,\bar s}(\mathcal{A})}:=\|\,f^{\circ}\,\|_{C^{\alpha}(\mathcal{A}^{\circ})}+\|\,\tilde{f}\,\|_{C^{\alpha}_{\bar s}({\mathcal{A}})}.$$
\end{defn}

\smallskip

\begin{remark} We observe that $ f(z,\bar x)\in\,C^{\alpha}_s({\mathcal{A}})$ iff the function
$\bar f(\xi,\bar x):= f(e^\xi,\bar x)$ belongs to the space $C^{\alpha}(\bar \A)$ (this is the 
H\"older space with respect to the standard metric) where $\bar A:= \{ \, (\xi,\bar x)\,\, : \,\,
(e^\xi,\bar x) \in \A \, \}.$
\end{remark}


We will next define weighted  H\"older spaces of second order derivatives with respect to our metric
$\bar s$. 

\begin{defn} [The space $C^2_{w}(\mathcal{A})$]\label{def-space}
We say that a continuous function $f$ on
$\mathcal{A}$ belongs to $C^2_{w}(\mathcal{A})$ if
$f^{\circ}\,\in\,C^2(\mathcal{A}^{\circ})$ and $f$ has continuous derivatives
 $$ f_z,\,\, f_{i},\,\, f_{zz},\,\, f_{zi},\,\, f_{ij}, \qquad i,j=2,\cdots,n$$
\no in the interior of $\mathcal{A}$  such that
$$\frac{1}{\sqrt{z}}\,(f-f^{\circ}),\,\,\sqrt{z}\,f_z,\,\frac{1}{\sqrt{z}}\,(\,f_{_i}-f_{_i}^{\circ}\,),\,
 z^{\frac{3}{2}}\,f_{zz},\, \sqrt{z}\,f_{zi},\,\frac{1}{\sqrt{z}}\,(\,f_{ij}-f_{ij}^{\circ}\,)$$
\noindent extend continuously (with respect to the standard euclidean  metric) up to the boundary $z=0$, for all $i,j=2,\cdots,n$. The norm of $f$ in the space $C^2_{w}(\mathcal A)$ is defined as follows:
$$\|\,f\,\|_{C^2_{w}(\mathcal{A})}:=\|\displaystyle\sum_{m=0}^2\,D^m_{\bar x}\,f^{\circ}\,\|_{C^{0}(\mathcal{A}^{\circ})}+
\displaystyle\sum_{m+n=0}^2\|\,z^m\, D_z^m\,D_{\bar x}^n\,\tilde f\,\|_{C^0({\mathcal{A}})}.$$

\end{defn}

\begin{defn}[The space $C^{2+\alpha}_{w,\bar s}(\mathcal{A})$]
Given $f\in C^{2}_{w}(\mathcal{A})$, we say that $f$ belongs
to $C^{2+\alpha}_{w,\bar s}(\mathcal{A})$ if
$$f^{\circ}\in
C^{2+\alpha}(\mathcal{A}^{\circ})\quad \mbox{and}\quad  z\, f_z, f_i,\,z^2\,f_{zz},  z\, f_{zi},f_{ij},
\qquad i,j=2,\cdots,n$$
 extend continuously up to the boundary  and the extensions
are H\"older continuous on $\mathcal{A}$ of class
$C_{w,\bar s}^{{\alpha}}(\mathcal{A})$. The norm of $f$ in the space
$C^{2+\alpha}_{w,\bar s}(\mathcal{A})$ is defined as:
$$\|\,f\,\|_{C^{2+\alpha}_{w, \bar s}(\mathcal{A})}:=
\|\,f^{\circ}\,\|_{C^{2+\alpha}(\mathcal{A}^{\circ})}+\displaystyle\sum_{m+n=0}^2\|\,z^{m}\,
D_z^m\,D_{\bar x}^n \,f\,\|_{C^{\alpha}_{w,\bar s}(\mathcal{A})}.$$

\end{defn}

\noindent Let $\tau >0$. Similarly as in \cite{CD}, the definitions above can be naturally extended
on a  space-time domain $Q_\tau=\A \times [0,\tau]$ by using the parabolic distance
$d{s}^2=d\bar s^2+|dt|$. We call the resulting spaces $C^{\alpha}_{w,s}(Q_\tau),
C^{2}_{w,s}(Q_\tau)$ and $C^{2+\alpha}_{w,s}(Q_\tau)$
respectively. 


\subsection{The Degenerate Equation on the disc}\label{deg}

\medskip
We will show in Section \ref{global} that the initial value problem \eqref{eqn-qk} can be  
transformed, via a  global coordinate change,  to an  initial value problem of the form \eqref{eqn-ivp}. Its linearization at a point
$\bar w$ close to the  initial date $w_0$ is a  degenerate equation of the form
\begin{equation}\label{eqn-ball}
Lw: = w_t  - \, (\, a^{ij} w_{ij} + b^i\, w_{i} + c\,w \, ) =0
\end{equation}
on the cylinder ${\mathcal D}  \times [0,\tau)$, $\tau >0$,  where
${\mathcal D}$ denotes the unit ball in  ${\R}^n$.
The matrix $\{a^{ij}\}$ is  symmetric and 
under appropriate an change of coordinates near $\partial \D$ which straightens  the
boundary, equation  \eqref{eqn-ball} is  transformed into the degenerate equation  of the form
~({\ref{eqn-hl}).

We define the distance 
function $s$ in ${\mathcal D}$ as follows:
in the interior 
of ${\mathcal D}$, $ s$ it is equivalent to
the standard euclidean distance,
while around any boundary point $P \in \partial 
\mathcal D$, $ s$ is defined as the pull
back of the distance function induced by the metric $d\bar s^2$ defined in \eqref{eqn-metric} 
on the half space $\mathcal S_0  = \{ (z,x_2,\cdots,x_n) :  z \geq 0  \}$,
via a map 
$\varphi: \mathcal S_0  \cap {\mathcal D} \to {\mathcal D}$ that flattens the boundary 
of the ball ${\mathcal D}$ near $P$. We denote by $ds^2$ the associated parabolic distance.

We can now define the spaces $C^\alpha_{w,\bar s}({\mathcal D})$ and $C^{2+\alpha}_{w,\bar s}({\mathcal D})$:
for a fixed and small number $\delta$ in $0 < \delta <1$,
we write  
$${\mathcal D} = {\mathcal D}_{1-\delta/2} \, \cup  ({\displaystyle\bigcup_l}\, 
\left ( {\mathcal D}_{\delta}(P_l) \cap {\mathcal D}) \right )$$
for finite many points $P_l \in \partial {\mathcal D}$, $l \in I$, 
with ${\mathcal D}_{1-\delta/2}$ denoting the ball centered at the origin of radius
$1-\delta/2$ and ${\mathcal D}_{\delta}(P_l)$ denoting the ball  of radius
$\delta$ centered at $P_l$. 

We denote by ${\mathcal D}_+$ the half disk
${\mathcal D}_+ = \{ \, (z,\bar x) \in {\mathcal D} : \,\, z \geq 0 \, \}.$
We can choose charts $\Upsilon_l : {\mathcal D}_+ \to {\mathcal D}_{\delta}(x_l) \cap 
{\mathcal D}$
which flatten the boundary of ${\mathcal D}$ and such that $\Upsilon_l (0) = P_l$, $ l\in\,I$. Let $\{\psi$, $\psi_l\}$ 
be a partition of unity subordinated to the cover
$\{ \, {\mathcal D}_{1-\delta/2},\,
( {\mathcal D}_{\delta}(P_l) \cap {\mathcal D}) \,\} $ of ${\mathcal D}$, with $l\in\, I$.

\begin{defn}[The spaces $C^{\alpha}_{w,\bar s}({\mathcal D})$ and $C^{2+\alpha}_{w,\bar s}({\mathcal D})$]
We define  $C^\alpha_{w,\bar s}({\mathcal D})$ 
to be the space of all functions $w$
on ${\mathcal D}$ such that $w \in C^{\alpha} ({\mathcal D}_{1-\delta/2})$ and 
$w \circ \Upsilon_l \in C^\alpha_{w, \bar s}({\mathcal D}_+)$
 for all $l \in I$. Also,  we  define $C^{2+\alpha}_{w,\bar s}({\mathcal D})$ to be the space  of all functions $w$
on ${\mathcal D}$ such that $w \in C^{2+\alpha }({\mathcal D}_{1-\delta/2})$ and
$w \circ \psi_l \in C^{2+\alpha}_{w,\bar s}({\mathcal D}_+)$ for all $l \in I$.
\end{defn}

In the above definition  $C^{\alpha}$ and $C^{2+\alpha}$ denote the regular H\"older Spaces, 
while $C^{\alpha}_{w,\bar s} ({\mathcal D}_+)$ and $ C^{2+\alpha}_{w,\bar s}({\mathcal D}_+)$ denote the
H\"older Spaces defined in Section~\ref{def:ban}.
One can show that both spaces $C^\alpha_{w,\bar s}({\mathcal D})$ and $C^{2+\alpha}_{w,\bar s}({\mathcal D})$ are Banach spaces 
under the  norms
$$ \|w\|_{C^\alpha_{w,\bar s}({\mathcal D})} = \|\psi \, w\|_{C^{\alpha} ({\mathcal D}_{1-\delta/2})}
+ \sum_{l} \, \|\psi_l \,( w \circ \Upsilon_l)\|_{C^\alpha_{w,\bar s}({\mathcal D}_+)}$$ 
and
$$ \|w\|_{C^{2+\alpha}_{w,\bar s}({\mathcal D})} = \|\psi \, w\|_{C^{2+\alpha}(
{\mathcal D}_{1-\delta/2})}
+ \sum_{l} \, \|\psi_l \,( w \circ \Upsilon_l)\|_{C^{2+\alpha}_{w,\bar s}({\mathcal D})}.$$
\noindent

The above definitions can be extended in
a straight forward manner to the parabolic spaces 
$C^\alpha_{w,s}(Q_\tau)$ and $C^{2+\alpha}_{w,s}(Q_\tau)$, where $Q_\tau$ is the 
cylinder $Q= {\mathcal D} \times [0,\tau]$, for some $\tau >0$.
\smallskip

Before we state the main result in this section, we will give the assumptions on the coefficients
of the equation \eqref{eqn-ball}
on the cylinder $Q = {\mathcal D} \times [0,T)$:  We first assume that for any  $\delta$ in $0<\delta <1$, the  
coefficients $\{a^{ij}\}$, $\{b^i\}$ and $c$
belong to the 
H\"older class $C^{\alpha}(\mathcal D_{1-\delta/2}
 \times [0,T])$, which means that the coefficients 
are of the class $C^{\alpha}$ in the interior of $\mathcal D$. In addition we assume that
the metric $\{ a_{ij} \}$ is strictly elliptic in $\mathcal D_{1-\delta/2}$. 
For a number $\delta$ in $0<\delta <1$,
let  $\Upsilon_l: \mathcal D_+ \to \mathcal D_\delta(P_l) \cap \mathcal D$
be the collection 
of charts which flatten the boundary of $\mathcal D$,
considered above.
We assume that there exists a number $\delta$
so that for every $l \in I$, the coordinate change 
introduced by each of the $\Upsilon_l$ transforms the  
operator $L[w]$ defined in \eqref{eqn-ball} 
on $\mathcal D_\delta(P_l) \cap \mathcal D$, 
into an operator $\widetilde L_l$ on $\mathcal D_+$ of the form
\eqref{eqn-hl}
with the coefficients 
$\tilde  a_{ij}, \, \tilde b_i, \, \tilde c$ 
belonging to the class  and with $\{ \tilde a_{ij} \}$ strictly elliptic. 

\begin{thm}\label{thm:exi2} Assume that the operator $L$ satisfies all the above conditions
on the cylinder $Q_\tau={\mathcal{D}} \times [0,\tau]$.
Then, given any function $w_0 \in C^{2+ \alpha}_{w,\bar s}({\mathcal{D}})$ and any function
$g \in C^\alpha_{w,s}(Q)$
there exists a unique solution $w \in C^{2+ \alpha}_{w,s} (Q_\tau)$
of the initial value problem
\[
\left\{
\begin{array}{ll}
 Lw =g
& in\,\,Q\\
w(\cdot,0)=w_0 & on\,\,\mathcal D%
\end{array}%
\right.
\]
satisfying
\begin{equation}\label{ineq:www}
\|w\|_{C^{2+\alpha}_{w,s}(Q)} \leq C(\tau) \left ( \|w_0\|_{C^{2+\alpha}_{w,\bar s}(\mathcal{D})} + \|g\|_{C^\alpha_{w,s}(Q)} \right ).
\end{equation}
The constant $C(\tau)$ depends  only on the 
numbers $\alpha$, $\tau$, the ellipticity constant of $\{ \tilde a_{ij} \}$ and the H\"older
norms of the coefficients. 
\end{thm}

\begin{proof}
The proof follows the arguments of the proof of Theorem 4 in ~\cite{CD}. 
\end{proof}

Finally, the following existence result follows from Theorem \ref{thm:exi2} and the Inverse
Function Theorem between Banach spaces.

\begin{thm}\label{thm:exi3}
Let $w_0$ be a function in   $C^{2+\alpha}_{w,\bar s}(\mathcal D)$.
 Assume that the linearization $DM(\bar w)$ of the 
fully-nonlinear operator
\begin{equation}\label{eqn-M}
Mw = w_t - F(t,u,v,w,Dw, D^2w)
\end{equation}
defined on the cylinder $Q = \mathcal D  \times [0,T]$,
satisfies the hypotheses of Theorem~\ref{thm:exi2}  at
all points $\bar w \in C^{2+\alpha}_{w,s}(Q)$, with 
$||\bar w - w_0||_{ C^{2+\alpha}_{w,s}(Q)} \leq \mu$, for some
$\mu >0$.
Then,  there exists a number $\tau_0$ in $0<\tau_0 \leq T$
depending on $\alpha$,  $\mu$ and the ellipticity of $\{\tilde a_{ij} \}$, for which the initial value problem
\begin{equation}\label{eqn-ivp2}
\left\{\begin{array}{ll} Mw=0 
\qquad &\text{in $\,\,\mathcal D \times [0,\tau_0]$}\\
w(\cdot,0)=w_0 \qquad &\text{on $\,\, \mathcal D$} \end{array}\right. 
\end{equation}
 admits a solution $w$ in the space $C^{2+\alpha}_{w,s}
(\mathcal D \times [0,\tau_0])$. Moreover,
$$||w||_{C^{2+\alpha}_{w,s}(\mathcal D \times [0,\tau_0])}
\leq C\, ||w_0||_{C^{2+\alpha}_{w,s}(\mathcal D)}$$
for some positive constant $C$ which depends only $\alpha$,  $\mu$ and the ellipticity of $\{\tilde a_{ij} \}$.\end{thm}

\begin{proof} The proof follows the arguments of the proof of Theorem 5 in ~\cite{CD}. 
\end{proof}

\subsection{Global change of coordinates and existence in $C^{2+\alpha}_{w,s}$}\label{global}

In this section we introduce a global change of coordinates which transforms the \eqref{eqn-qk}  flow for a surface $\Sigma_0$ with flat sides  into a fully-nonlinear degenerate parabolic PDE on 
the unit ball $\mathcal D$. We will only give a brief outline,  as  all the results  
are a straightforward generalization of the  2-dimensional case treated in  ~\cite{CD}.

We recall that our initial surface $\Sigma_0$ is of the form $\Sigma_0 = \Sigma_0^1 \cup \Sigma_0^2$ with $\Sigma_0^1$ flat and $\Sigma_0^2$ strictly convex.  Let $\Si_0$ be a smooth surface with boundary   close to $\Sigma^2_0$ and such that $\partial \Si_0 \subset \{ x_{n+1} =0 \}$. Let $S_0 : \mathcal D \to \mathbb{R}^{n+1}$, indicate a
parameterization of $\Si_0$ on the unit disk $\mathcal D$, namely if  $U=(u_1,\cdots,u_n)$ and $X=(x_1,\cdots,x_n)$ we have 
 $$S_0(U)=(X,  z)\in \R^{n+1}\,\,\,\hbox{ which maps }\,\,\,\partial \mathcal D\,\,\, \hbox{ onto}\,\,\, \partial \Si_0.$$ 

\noindent Let $\eta >0$ be sufficiently small. Let $T= (T_1,T_2,\cdots, T_{n+1})$ be a
smooth vector field transverse to $\Si_0$ which is parallel to the $x_{n+1}=0$
plane in a small neighborhood near $\partial \Si_0$.  Following \cite{CD}, we define the global change of
coordinates $\varPhi  : \mathcal D \times [-\eta,\eta] \to \mathbb{R}^{n+1}$ by
\begin{equation}\label{eqn-coc}
\left(
\begin{matrix}
 X  \\ z\end{matrix} \right) = \varPhi  \left(\begin{matrix} U \\  w\end{matrix}
 \right)= S (U) + w\, T(U).
\end{equation}

\smallskip

\begin{thm}\label{thmreg}
Assume that the initial surface $\Sigma_0$ satisfies the assumptions of
Theorem~\ref{thm-main2}. Then, there exists a time $\tau >0$ 
for which the solution $\Sigma_t$    of the  \eqref{eqn-qk}  flow is converted, via
the coordinate change \eqref{eqn-coc},    to a solution 
$w(\cdot, t)$ of the IVP \eqref{eqn-ivp2}  on $Q_\tau = \mathcal D \times
[0,\tau]$  with  $w_0 \in C^{2+\alpha}_{w,s}(\mathcal D)$ and
the operator $M$ satisfying all the hypotheses of Theorem  \ref{thm:exi3}. 

\end{thm}

\begin{proof} The proof of this result follows along the lines of the proof of Theorem 6 
in \cite{CD}. It is clear that the coordinate change \eqref{eqn-coc} transforms the free-boundary
problem  \eqref{eqn-qk} into a   problem of the form \eqref{eqn-ivp2} where the fully-nonlinear
operator $M$ is strictly elliptic away from the lateral boundary $\partial \D \times [0,\tau]$.
The main difficulty here is to show that the linearization of $M$ at a point $\bar w$ close to
the initial data $w_0$ is a degenerate operator of the the form \eqref{eqn-ball} which 
near the boundary lateral boundary $\partial \D \times [0,\tau]$ satisfies the assumptions of
Theorem \ref{thm:exi2}. In the two-dimensional case, this is done in detail in \cite{CD} (Theorem 6).
In Section \ref{sec-appendix} we will show that the linearization of equation \eqref{eqn-ff}, which 
is obtained from \eqref{eqn-qk} via a local change of coordinates that fixes the free-boundary, 
is of the desired from. The computations for the global change of coordinates are more involved
by similar.

\end{proof}

As an immediate consequence of Theorems \ref{thm:exi2} and   \ref{thm:exi3} we obtain
the following existence result for \eqref{eqn-qk}.

\begin{thm}\label{thm:exis4}
Assume that the initial surface $\Sigma_0$ satisfies the assumptions of
Theorem~\ref{thm-main2}. Then, there exists a time $\tau >0$ 
for which there exists a  solution $\Sigma_t$  of \eqref{eqn-qk} on $[0,\tau]$ which
is of the class $C^{2+\alpha}_{w,s}$. 
\end{thm}

Once a $C^{2+\alpha}_{w,s}$ solution of \eqref{eqn-qk} is established,  one can argue
similarly as in \cite{CD} that the pressure $g(\cdot,t)=\sqrt{u(\cdot,t)}$ of the solution $\Sigma_t$  (as defined in the introduction)  is  $C^\infty$ smooth up to the interface. This concludes the regularity
part of Theorem \ref{thm-main2},  which  we state next.

\begin{thm}\label{thm:exis5} Assume that the initial surface $\Sigma_0$ satisfies the assumptions of
Theorem~\ref{thm-main2}.  Let $\Sigma_t$ be the unique 
viscosity solution of \eqref{eqn-qk} with initial data $\Sigma_0$ (its existence follows from  Theorem \ref{thm-main1}).  
Then,
there exists a time $\tau>0$  such that  the pressure function $g(\cdot,t)=\sqrt{u(\cdot,t)}$  is smooth up to the interface $x_{n+1}=0$ and satisfies condition~($\star$).
\end{thm}

\subsection{Proof of Theorem \ref{thm-main2}}\label{sec-th2}

The proof of  Theorem \ref{thm-main2} readily follows from Theorem \ref{thm:exis5} and 
 the following proposition that refers to the evolution of the boundary of the flat side.

\begin{prop} \label{prop-natasa} Under the assumption that $\Sigma_t$ belongs to the 
class $\mathfrak S$ (as defined in Definition  \ref{defn-surface})  and   that the non-degeneracy 
condition~($\star$) holds on  $0 \leq t \leq \tau$, then the boundary
of the flat side $\Gamma_t$ is an {\em ($n-1$)}-dimensional surface which  evolves by the {\em ($*_{k-1}$)} flow.
\end{prop}

\begin{proof}
Let $P_0 \in \Gamma_t$ be a point on the boundary of the flat side.  The strictly
convex part of the surface $\Sigma^2_t$  can be expressed locally around $P_0$ 
as the graph of $z = u(x_1,\dots,x_n,t)$, where we may assume that  the coordinates are
chosen so that $x_2,\dots,x_n$ are the tangential directions to the flat side at 
$P_0$. We consider, as before, the pressure function  $g = \sqrt{u}$ and we solve the equation $z = u(x_1,\dots,x_n,t)$
with respect to $x_1$. Then $x_1 = f(z,x_2,\dots,x_n,t)$.

We observe that because of the non-degeneracy condition~($\star$), we have 
\begin{equation}
f_{x_i}(P_0) = 0, \qquad  \forall \,\,\,   i \geq 2.
\end{equation}
Indeed,  along the tangential directions $x_i$  ($i \geq 2$)  to $\Gamma_t$ at $P_0$,     we have 
$$f_{x_i} = -\frac{u_{x_i}}{u_{x_1}} = - \frac{g\, g_{x_i}}{g\, g_{x_1}} = -\frac{g_{x_i}}{g_{x_1}}=0,
\qquad \mbox{at} \,\,\, P_0$$
because  $g_{x_i}(P_0)=0$, for $i \geq 2$ while   $g_{x_1}(P_0) > 0$ due to  ($\star$). 

The assumptions   that $\Sigma_t \in \mathfrak S$ and that the non-degeneracy
condition holds, imply  in particular that  $f \in C^{2+\alpha, \frac 12}$ near $P_0$.
Hence, by  Definition \ref{def-space}  the functions 
$\frac{1}{\sqrt{z}}\, f_{_i}$ extend continuously  up to the boundary 
$z = 0$, when $i \geq 2$, which means that
\begin{equation}\label{claim-first-der}
|f_{x_i}| \leq C\, \sqrt{z},  \qquad  \mbox{as } \,\,\,   z \to P_0, \quad  \forall \,\,\,  i \geq 2.\end{equation}

Our surface is locally expressed as a graph  $z=u(x_1,\dots,x_n,t)$
and its principal curvatures are the eigenvalues of the symmetric
matrix $[a_{ij}]$, where
\begin{equation}
\label{eq-aij}
a_{ij} = \frac{1}{v}\, \left ( D_{ij}u - \frac{D_iuD_luD_{jl}u}{v(1+v)} - 
\frac{D_juD_luD_{il}u}{v(1+v)} + \frac{D_iuD_juD_kuD_luD_{kl}u}{v^2(1+v)^2}\right )
\end{equation}
and $v = \sqrt{1+|Du|^2}$. The equation (\ref{eqn-qk}) can be
expressed as
$$u_t = \frac{S_k^n([c_{ij}])}{S_{k-1}^n([c_{ij}])}, \qquad \mbox{with} \quad c_{ij} = a_{ij}\, v.$$
Moreover,  since $f_t =-{u_t}/{f_z}$, 
the equation (\ref{eqn-qk}) can be written in terms of $f$ as
\begin{equation}\label{eq-f}
f_t = -\frac{S_k^n([b_{ij}])}{S_{k-1}^n([b_{ij}])}, \qquad \mbox{with} \quad b_{ij} = c_{ij}\, f_z
\end{equation}
where $b_{ij}$ can be expressed in terms of $f$ and
its first and second derivatives. 

To express the evolution equation in
terms of $f$ we use the identities
\begin{equation*}
\label{eq-id1}
u_{x_1} = \frac{1}{f_z}, \qquad  u_{x_i} = -\frac{f_{x_i}}{f_z},   \qquad 
u_t = -\frac{f_t}{f_z}
\end{equation*}
and
\begin{equation*}
\label{eq-id2}
u_{x_1x_1} = -\frac{f_{zz}}{f_z^3}, \qquad  u_{x_1x_i} = -\frac{1}{f_z} \left (-\frac{f_{x_i}}{f_z^2}f_{zz}
+ \frac{1}{f_z}f_{zx_i} \right )
\end{equation*}
and
\begin{equation*}
\label{eq-id3}
u_{x_ix_j} = -\frac{1}{f_z} \left (\frac{f_{x_i}f_{x_j}}{f_z^2}f_{zz} - \frac{f_{x_i}}{f_z}f_{zx_i}
- \frac{f_{x_j}}{f_z}f_{zx_j} + f_{x_ix_j}\right )
\end{equation*}
where $i, j > 1$.

Keeping in mind Definition \ref{def-space} and that we are interested
in a behaviour of $[b_{ij}]$ around $z = 0$, using (\ref{eq-aij}) and the above formulas 
we can compute that, as $z \to 0$, we have 
$$
b_{11} = -\frac{f_{zz}}{f_z^2} + o(1), 
\qquad b_{1i}=b_{i1}= -\frac{f_{zx_i}}{f_z}+\frac{f_{x_i}f_{zz}}{f_z^2} + o(1)
$$
and
$$
b_{ij} = {-f_{x_ix_j}+\frac{f_{x_i}f_{zx_j}}{f_z}+\frac{f_{x_j}f_{zx_i}}{f_z}-\frac{f_{x_i}f_{x_j}f_{zz}}{f_z^2} + o(1)}$$
for $i,j \geq 2$. 

To simplify the notation we set from now on $\bar x=(x_2,\dots,x_n)$.  By  Definition \ref{def-space} we can see that the behaviour of $[b_{ij}]$, as $z \to 0$,  is 
\begin{equation*}
b_{11} = \frac{c_1(z,\bar x,t)}{\sqrt{z}}, \quad b_{1i}=b_{i1}=c_i(z,\bar x, t), \quad 
b_{ij} = - f_{ij} + c_{ij}(z,\bar x, t), \quad  i,j \geq 2
\end{equation*}
with 
$$|c_i(z,\bar x,t)| \leq C \qquad \mbox{and} \quad \lim_{z\to 0}\, c_{ij}(z,\bar x,t) = 0$$
and in addition, by the non-degeneracy condition ($\star\star$),  
$$0 < \delta  \le c_1(z,\bar x, t) \le C.$$

We may assume that the  coordinates $\bar x:= (x_2, \dots, x_n)$ are chosen so that the matrix 
$[f_{ij}(P_0)]_{i,j \geq 2}$
is diagonal at our chosen boundary $P_0$. Then, it follows from the above discussion
that 
\begin{equation}\label{eqn-fff}
b_{ii} = -f_{ii} + o(1), \qquad (b_{ij})_{i \neq j}  = o(1),\qquad  i,j \geq 2.
\end{equation}
The  non-degeneracy condition ($\star\star$) and our regularity assumption on $f$  imply  that 
\begin{equation}\label{eqn-ff2}
0 < \lambda \leq  b_{ii} \leq C < \infty, \qquad (b_{ij})_{i \neq j}  = o(1),\qquad  i,j \geq 2.
\end{equation}
It follows that the  eigenvalues of $[b_{ij}]$ are computed as the roots of a   polynomial in
$\lambda$ which is of the form 
\begin{equation}
\label{eq-eigenvalue}
\det([b_{ij}]-\lambda\,  I) = \left ( \frac{c_1}{\sqrt{z}}\, \Pi_{j=2}^n (b_{jj}-\lambda)  - 
\sum_{i=2}^n c_i^2\, \Pi_{j\neq i, j\geq 2}(b_{jj}-\lambda) \right )  (1 + o(1)).
\end{equation}
Denote the eigenvalues of the matrix $[b_{ij}]$  by 
$$\lambda_1(z,\bar x, t) \ge \dots \ge \lambda_n(z,\bar x, t).$$

\begin{claim}
\label{claim-eigen-beh}
There are uniform constants $\mu > 0$ and $\nu < \infty$ so that 
\begin{equation}\label{eqn-lala}
 \mu \leq \frac {\lambda_1 (z,\bar x, t) }{\sqrt{z}},  \, \lambda_i (z,\bar x, t) \leq \nu, \qquad i \geq 2.
 \end{equation}
\end{claim}

\begin{proof}
We will first show there is a uniform constant $\mu > 0$ so that
\begin{equation}
\label{eq-low-bound}
\lambda_i \ge \mu, \qquad  \mbox{for all} \,\,\, i\ge 1.
\end{equation}
Assume, by contradiction, that   there is a $\lambda_k$
so that $\lambda_k(z,\bar x, t)\to 0$,  as $z \to 0$.  Since,  $\det([b_{ij}]-\lambda_k\,  I)=0$,  by applying \eqref{eq-eigenvalue} with  $\lambda=\lambda_k$ and using   \eqref{eqn-ff2}  we  obtain
\begin{equation}
\label{eq-id-zero}
\lim_{z \to 0} \left (  \frac{c_1(z,\bar x, t)}{\sqrt{z}}\, \Pi_{i=2}^n \, b_{ii}(z,\bar x, t) - 
\sum_{i=2}^nc_i(z,\bar x, t)^2\,  \Pi_{j\neq i,j\geq 2}\, b_{jj}(z,\bar x, t) \right )=0
\end{equation}
which is impossible by \eqref{eqn-ff2} and the fact that all $b_{ii}$, $c_i$ are bounded and $c_1 \geq \delta >0$. 
Also, note that
$$\Pi_{i=1}^n\lambda_i = \det([b_{ij}])  = \left ( \frac{c_1}{\sqrt{z}}\, \Pi_{j > 1} b_{jj}
- \sum_{i>1}c_i^2 \, \Pi_{i\neq j, j \geq 2} b_{jj} \right )(1+ o(1))$$
and 
$$\sum_{i=1}^n\lambda_i= \frac{c_1}{\sqrt{z}} + b_{22} + \dots + b_{nn} + o(1)$$
are both of the order $z^{-\frac{1}{2}}$ as $z\to 0$. Hence by  (\ref{eq-low-bound}) and \eqref{eqn-ff2}
we conclude the bounds \eqref{eqn-lala} finishing the proof of our claim. 
\end{proof}

\begin{claim}
\label{claim-fiilambda}
After a possible index re-arragement we have
$$\lim_{z \to 0} \, ( f_{ii}(z,\bar x,t) - \lambda_i(z,\bar x,t) ) =0, \qquad \forall i \geq 2.  $$
\end{claim}

\begin{proof}
Let  $\lambda_k$, $k\ge 2$ be any eigenvalue. Then, using \eqref{eq-eigenvalue}  we have  
$$\lim_{z \to 0} \left ( \frac{c_1}{\sqrt{z}}\, \Pi_{j\ge 2}(b_{jj}-\lambda_k) - \sum_{i\ge 2} c_i^2\, \Pi_{j\neq i, j\ge 2}(b_{jj}-\lambda_k) \right ) =0$$
and since the second term  is bounded, we conclude using also \eqref{eqn-ff2} and \eqref{eqn-fff} that 
\begin{equation}\label{eqn-lalala}
\lim_{z \to 0} \, ( f_{ii}(z,\bar x,t) - \lambda_k(z,\bar x,t) ) =0
\end{equation}
for some $i \geq 2$. 

It remains to show that, vise versa, for every $i \geq 2$, there exists an eigenvalue $\lambda_k$,
$k \geq 2$,  so that \eqref{eqn-lalala} holds.  We will argue by contradiction. 
Assume that there exists an $i \geq 2$ so that  \eqref{eqn-lalala} fails to hold for all $\lambda_k$,
$k \geq 2$. Without loos of generality, lets assume that $i=n$, so that 
$$\lim_{z\to 0}(f_{nn}(z,\bar x,t) - \lambda_k(z,\bar x,t)) \neq 0, \qquad \forall \, k \geq 2.$$

We will look at an  eigenvector $V:=(v_1,\cdots, v_n)$  corresponding
to any  $\lambda\in \{\lambda_2, \dots, \lambda_n\}$.  Using 
our previous computations on the behavior of the matrix $[b_{ij}]$, it follows that  the  coordinates of 
$V$ satisfy the following system of equations
\begin{eqnarray*}
& &\frac{c_1}{\sqrt{z}}\, v_1 + c_2\, v_2 + \dots + c_n\, v_n + o(1) = \lambda\, v_1 \\
& & c_j\, v_1 + f_{jj}\, v_j + o(1) = \lambda\,  v_j, \qquad \forall \,\,  j \ge 2.
\end{eqnarray*}
Since $\lambda$ and all the coefficients $c_i(z,\bar x,t)$ are bounded as $z\to 0$, the first equation implies
the $\lim_{z\to 0}v_1(z,\bar x,t) = 0$. The last equation gives that
$$c_n\, v_1 = (\lambda - f_{nn})\, v_n + o(1), \qquad \mbox{as} \,\,\, z \to 0$$ and since
the $\lim_{z\to 0} (f_{nn} - \lambda)  \neq 0$ it implies that the $\lim_{z\to 0}\, v_n(z,\bar x,t) = 0$.
We conclude that all the
eigenvectors corresponding to $\lambda_2, \dots, \lambda_n$ are of the form
$$V^i = (o(1),v_2^i,\dots, v_{n-1}^i, o(1)), \qquad \mbox{as} \,\,\, z \to 0.$$ 
We will argue that this is impossible because the $V_i's$ at the limit  $ z\to 0$ must span  the tangent plane  to the
surface $z=f(0,\bar x,t)$,  which is $(n-1)$-dimensional. 

To this end, we  consider the slices $x_1 = f(z,\bar x,t)$, when $z$ is
fixed, but close to zero.  Since $x_2, \dots, x_n$ are the tangential
directions to the flat side at $P_0$ and since the slices $x_1 =
f(z,\bar x,t)$ converge nicely to the flat side $x_1 =
f(0,\bar x,t)$ (from our regularity assumptions on $f$), it follows that  $x_2,\dots, x_n$  are almost tangential
directions to the slice $x_1 = f(z,x_2,\dots,x_n,t)$, when $z$ is
close to zero. Therefore the eigenvectors $V^2,\dots,V^n$ span an
$(n-1)$-dimensional plane that is almost tangent to the graph $x_1 =
f(z,x_2,\dots,x_n,t)$. Because of the nice convergence of the slices
to the flat side, those almost tangent planes converge to the tangent
plane to the interface $x_1=f(0,\bar x,t)$  at $P_0$. 

On the other hand, we observe that each $V^i$ converges, as
$z\to 0$,  to a vector of the form 
$$
\bar V^i = (0,\bar v_2^i,\dots,\bar v_{n-1}^i,0).
$$
The span of $\langle V^2,\dots,V^n\rangle$ converges to a span
of $\langle \bar V^2,\dots,\bar V^n \rangle$,  which
is at most $(n-2)$-dimensional and therefore it is impossible to  define the  tangent
plane to $x_1 = f(0,x_2,\dots,x_n,t)$. This finishes the proof of our claim.
\end{proof}

\begin{claim}
\label{claim-eigen-interface}
The principal curvatures   of the interface $x_1 = f(0,x_2,\dots,x_n,t)$
are given  by $$ \bar  \lambda_i = f_{ii}(0,x_2,\dots,x_n,t), \qquad  i\ge 2.$$
\end{claim}

\begin{proof}
Since the  interface is the graph of the function  $x_1=f(0,x_2,\dots,x_n,t)$,
its  principal curvatures  can be
computed by using formula (\ref{eq-aij}). Since  $\nabla f(P_0) = 0$,
the principal curvatures of the
interface are the eigenvalues of the matrix
$[f_{ij}]$. By our choice of coordinates at $P_0$
this matrix is diagonal, which proves the claim.
\end{proof}

We will now conclude the proof of Proposition \ref{prop-natasa}. Denote by $\bar \lambda_i(\bar x,t)$ the principal curvatures of the
interface $z=0$.  By Claims
\ref{claim-fiilambda}  and 
\ref{claim-eigen-interface}, we have 
\begin{equation*}
\label{eq-sets}
\bar \lambda_i (\bar x,t)  = \lim_{z\to 0}\lambda_i(z,\bar x,t), \qquad \forall i \geq 2.
\end{equation*}
Since $\lim_{z \to 0} \lambda_1(z,\bar x,t) = \infty$, by  L'Hospital's rule, we obtain
\begin{eqnarray*}
f_t(0,\bar x,t) &=& -  \lim_{z\to 0}\frac{S_k^n([b_{ij}]}{S_{k-1}^n([b_{ij}])} 
= - \lim_{z\to 0} \frac{\frac{\partial}{\partial\lambda_1}S_k^n([b_{ij}])}
{\frac{\partial}{\partial\lambda_1}S_{k-1}^n([b_{ij}])} \\
&=& -  \lim_{z\to 0}\frac{S_{k-1}^{n-1}(\lambda_2,\dots,\lambda_n)}{S_{k-2}^{n-1} (\lambda_2,\dots,\lambda_n)} = - \frac{S_{k-1}^{n-1}(\bar \lambda_2,\dots,\bar \lambda_n)}{S_{k-2}^{n-1}
(\bar \lambda_2,\dots,\bar \lambda_n)}
\end{eqnarray*} 

\noindent which shows that  the interface $\Gamma_t$ shrinks  by the \mbox{($*_{k-1}$)} flow.
\end{proof}

We finally remark that we have   actually shown   the following stronger result, where
we relax the regularity assumptions on the initial surface.

\begin{thm}\label{thm6} Assume that the initial surface $\Sigma_0$  belongs to the class  $C^{2+\alpha,\frac{1}{2}}_s$ and satisfies the 
non-degeneracy conditions~($\star\star$). Then, there exists a $\tau >0$ for which the  \eqref{eqn-qk} flow with initial data the surface $\Sigma_0$
admits a solution $\Sigma_t$ which is smooth up to the interface,
for $0<t\leq \tau$. In particular, the interface $\Gamma_t$ is a 
smooth hypersurface for every $0<t\leq \tau$ which moves by the \em{($*_{k-1}$)} flow.
\end{thm}



\subsection{Appendix}\label{sec-appendix}

In this appendix we will justify why   the linearization of the equation \eqref{eq-f}
satisfied by $x_1=f(z,x_2,\cdots,x_n)$ 
is of the form (\ref{eqn-hl}). 

\begin{prop}
\label{lem-lin-eq}
The linearization of (\ref{eq-f}) around a point $\tilde f \in C^{2+\alpha}_{w,s}$
which satisfies the non-degeneracy condition ($\star\star$)  is of the form 
$$\tilde{f}_t =z^2\,  \tilde a_{11}  \tilde{f}_{11} + 2 \, z\, \tilde a_{1i}  \tilde{f}_{1i} + 
\tilde a_{ij} \tilde{f}_{ij} +z \tilde b_{1} \tilde{f}_{1} + \tilde b_i \,  \tilde{f}_{i} + \tilde c \, \tilde{f},\qquad 
 i,j \neq 1,$$
where $[\tilde{a}_{ij}]$ is a positive definite matrix.
\end{prop}

\begin{proof}
The proof of the proposition   relies on a  computation done with mathematica and we 
will just briefly outline its  steps. Let the linearization  of (\ref{eq-f}) around a point $f \in 
 C^{2+\alpha}_{w,s}$ which satisfies the non-degeneracy condition ($\star\star$), be
$$\tilde{f}_t = a_{11} \tilde{f}_{11} + 2a_{1i}  \tilde{f}_{1i} + 
a_{ij} \tilde{f}_{ij} + b_{1} \tilde{f}_{1} + b_i \,  \tilde{f}_{i} + c \, \tilde{f},\qquad 
 i,j \neq 1.$$
Notice that the linearized coefficients are given by
\begin{equation}
\label{eq-lin-coeff}
a_{ij} = \sum_{p=1}^n\frac{\partial Q_k}{\partial\lambda_p}\, \frac{\partial\lambda_p}{\partial f_{ij}},
\qquad 
b_i = \sum_{p=1}^n\frac{\partial Q_k}{\partial\lambda_p}\,\frac{\partial\lambda_p}{\partial f_i},
\qquad c = \sum_{p=1}^n\frac{\partial Q_k}{\partial\lambda_p}\,\frac{\partial\lambda_p}{\partial f}.
\end{equation}
It  has been computed in \cite{D} that
\begin{eqnarray*}
\frac{\partial Q_k}{\partial\lambda_p} &=& \frac{1}{S_{k-1}^n(\lambda)^2}(S_{k-1,i}^n(\lambda)^2 - S_{k,i}^n(\lambda)S_{k-2,i}^n(\lambda)) \\
&\ge& \frac{n}{k(n-k+1)}\, \left (\frac{S_{k-1,i}^n(\lambda)}{S_{k-1}^n(\lambda)} \right )^2
\end{eqnarray*}
where $S_{k,i}^n(\lambda)$ denotes the sum of all terms of $S_k^n(\lambda)$ not containing the factor $\lambda_i$.
It follows by  the Claim \ref{claim-eigen-beh} and the above inequality that  there are uniform constants $\tilde{C}_1, \tilde{C}_2 > 0$ so that
$$\frac{\tilde{C}_1}{\lambda_1^2} \le \frac{\partial Q_k}{\partial\lambda_1} \le \frac{\tilde{C}_2}{\lambda_1^2} \qquad \mbox{and} \qquad   \tilde{C}_1 \le \frac{\partial Q_k}{\partial\lambda_p} \le \tilde{C}_2, \quad  \forall p\ge 2.$$
Hence, by  Claim \ref{claim-eigen-beh}, we have  
\begin{equation}
\label{eq-coeff-speed}
C_1\, z \le \frac{\partial Q_k}{\partial\lambda_1} \le C_2\,  z \qquad \mbox{and} \qquad 
C_1 \le \frac{\partial Q_k}{\partial\lambda_p} \le C_2, \quad  \forall p\ge 2
\end{equation}
for uniform  $C_1, C_2 > 0$. Since $f \in C^{2+\alpha}_{w,s}$ and  satisfies the non-degeneracy condition ($\star\star$),
it follows by  Definition \ref{def-space} that 
\begin{equation}
\label{eq-beh-many}
 \sqrt{z}\, f_z  = a(z), \qquad    z^{\frac 32}\, f_{zz} = - v(z)
\end{equation}
and
\begin{equation}
\label{eq-beh-many2}
\frac 1{\sqrt{z}}\, f_{x_i} =  h_i(z),  \quad    \sqrt{z}\, f_{zx_i} = u_i(z), \quad  f_{x_ix_i} = -C_i(z),
\qquad 2\le i \le n
\end{equation}
where $v(z), a(z), C_i(z), h_i(z), u_i(z)$ are continuous functions at $z = 0$
and 
\begin{equation}
\label{eq-nozero}
v(z), a(z), C_i(z) \geq \delta >0.
\end{equation} 

\begin{claim}
\label{claim-a11}
The coefficient $a_{11}$ satisfies $a_{11} = z^2\, \tilde{a}_{11}\, (1+o(1))$,
where $\tilde{a}_{11}(z)$ is a strictly positive and continuous
function in a neighborhood of $z=0$.
\end{claim}

\begin{proof}
By (\ref{eq-coeff-speed}) we have
$$a_{11} = \left (z\,\frac{\partial\lambda_1}{\partial f_{11}} + \sum_{p\ge 2}C_i\, 
\frac{\partial\lambda_p}{\partial f_{11}} \right )\, (1+o(1))$$
where $C_i(z) \ge C_i > 0$ in a  neighborhood  of $z = 0$.
To compute $\frac{\partial \lambda_p}{\partial f_{lk}}$ we will use the fact that  $\lambda_p$ are
the eigenvalues of $[b_{ij}]$, namely 
\begin{equation}\label{eqn-above1}
\det(b_{ij} - \lambda_p\, I_{ij}) = 0
\end{equation}
where the matrix $[b_{ij}]$ is defined in section  \ref{sec-th2}.
If we differentiate \eqref{eqn-above1} with respect to $f_{lk}$ we get
\begin{equation}
\label{eq-formula}
\frac{\partial \lambda_p}{\partial f_{lk}} = \frac{\sum_{i,j} M^{ij}\frac{\partial b_{ij}}{\partial f_{lk}}}
{\sum_i M^{ii}}
\end{equation}
where $M^{ij}$ denotes  the $ij$-th minor of the matrix $[b_{ij} - \lambda_pI_{ij}]$.
A direct calculation yields to
$$\frac{\partial b_{11}}{\partial f_{11}} = -\frac{1}{f_z^2} + O(z^3), \qquad  
\frac{\partial b_{1i}}{\partial f_{11}} = \frac{f_{x_i}}{f_z^2} + O(z^3), \quad  i\ge 2$$
and 
$$\frac{\partial b_{jj}}{\partial f_{11}} = -\frac{f_{x_j}^2}{f_z^2} + O(z^3), \quad  j\ge 2, \qquad 
 \frac{\partial b_{ij}}{\partial f_{11}} = -\frac{f_{x_i}f_{x_j}}{f_z^2} + O(z^3), \quad  i\neq j \ge 2.$$
Moreover, by (\ref{eq-formula}) and (\ref{eq-beh-many})-(\ref{eq-beh-many2}) we have 
\begin{equation}
\label{eq-a1}
\frac{\partial\lambda_1}{\partial f_{11}} = -\frac{z}{a^2(z)}\, (1+o(1)).
\end{equation}
We claim that for $j\ge 2$,
$$b_{jj} - \lambda_j = e_j(z)\,\sqrt{z}, \,\,\, e_j(z) \ge 0 \qquad
\mbox{in a neighborhood  of} \,\,\, z = 0.$$  
To this end, we  notice that as in (\ref{eq-eigenvalue}) we have
\begin{equation}
\label{eq-det0}
\frac{c_1(z)}{\sqrt{z}}\,\Pi_{j\ge 2}(b_{jj}-\lambda_p) = 
\sum_{i\ge 2}c_i^2\,\Pi_{j\neq i, j\ge 2}(b_{jj}-\lambda_p)\,
(1+o(1))
\end{equation}
where $c_1(z) > 0$ in  a neighborhood of $z = 0$ and $p \ge 2$. We
may assume $b_{ii}(z) \neq b_{pp}(z), \,\,\, i\neq p$, since otherwise
we can divide (\ref{eq-det0}) by $(b_{pp}-\lambda_p)^k$, if there are
$k+1$ indices $i_i,\dots,i_{k+1}$ so that $b_{i_si_s} = b_{pp}$ for $1
\le s \le k+1$. We can rewrite (\ref{eq-det0}) as
$$\frac{c_1(z)}{\sqrt{z}} - \sum_{2 \le i\neq p}\frac{c_i^2}{b_{ii}-\lambda_p} = \frac{c_p^2}{b_{pp}-\lambda_p}.$$
Since all the terms in the sum on the left hand side are bounded and since the coefficient $ c_1(z) \ge \delta > 0$,
the left hand side is strictly positive in a  neighborhood of $z = 0$. This implies the bound 
$b_{pp} - \lambda_p \ge 0$ in a  neighborhood  of $z = 0$. By direct calculation we have 
\begin{equation}
\label{eq-a2}
\frac{\partial\lambda_p}{\partial f_{11}} = -\frac{z^2}{a(z)^2\, v(z)}
(a(z)^2\, e(z) + 2a(z)u_i(z)h_i(z) + 3h_i^2(z)v(z))\, (1+o(1)).
\end{equation} 
We claim $a(z)u_i(z)h_i(z) \ge 0$ in a  neighborhood of $z=0$. To this end, we 
recall  that by (\ref{eq-nozero}), we have $a(z) \ge \delta $ and that by \eqref{eq-beh-many2}
we have 
$$u_i(z)h_i(z) = \sqrt{z} f_{zx_i}\, \frac{f_{x_i}}{\sqrt{z}} = \frac{(f_{x_i}^2)_z}{2}.$$
Since $f_{x_i}^2(0,x_2,\dots,x_n,t) = 0$, $f_{x_i}^2 \ge 0$ and since
$u_i, h_i$ are continuous functions at $z = 0$,  we conclude that  $u_i(z)h_i(z)
\ge 0$ in a  neighborhood of $z=0$. 

Combining 
(\ref{eq-lin-coeff}), (\ref{eq-coeff-speed}), (\ref{eq-a1}) and (\ref{eq-a2}), we conclude that
$$a_{11} = \frac{z^2 \, [v(z) + \sum_{i\ge 2}(2a(z)^2e_i(z) + 
2a(z)u_i(z)h_i(z) + 3v(z)h_i^2(z))]}{a(z)^2\, v(z)}\, (1+o(1))$$
that is, $a_{11} = \tilde{a}_{11}\, z^2$, where 
$$\tilde{a}_{11}(z) = \frac{(v(z) + \sum_{i\ge 2}(2a(z)^2e_i(z) + 
2a(z)u_i(z)h_i(z) + 3v(z)h_i^2(z)))}{a^2(z)v(z)}\, (1+o(1))$$  
and hence
$$\tilde{a}_{11}(z) \ge \frac{1}{2a^2(z)}.$$
\end{proof}

\begin{claim}
The coefficients $a_{ii}$, $i\ge 2$,  are continuous and satisfy  the lower bound $a_{ii}(z) \ge \delta > 0$ in a neighborhood of $z=0$. 
\end{claim}

\begin{proof}
To simplify the notation, let us assume that $i=2$. Similarly as before, we have 
\begin{equation}
\label{eq-lin-a22}
a_{22} = \sum_{p\ge 1}\frac{\partial Q_k}{\partial\lambda_p}\,
\frac{\partial\lambda_p}{\partial f_{22}}.
\end{equation}
A direct calculation shows that
$$\frac{\partial b_{jj}}{\partial f_{22}} = -(1+o(1)), \quad \mbox{for}  \,\, j\ge 2, \qquad
\frac{\partial b_{ij}}{\partial f_{22}} = o(1) \quad  \mbox{in all other cases}.$$
Similar analysis as before yields to 
\begin{equation}
\label{eq-asymp-beh}
\frac{\partial\lambda_1}{\partial f_{22}} = O(z), \qquad  
\frac{\partial\lambda_p}{\partial f_{22}} = -(1+o(1)), \quad  p\ge 2.
\end{equation}
By (\ref{eq-coeff-speed}), (\ref{eq-lin-a22}) and (\ref{eq-asymp-beh})  we get
$$a_{22}(z) = z\, O(z) + \sum_{p\ge 2}\eta_p(z)(1+o(1))$$
where $\eta_p(z) \ge C_1 > 0$. This immediately implies the claim.
\end{proof}

The behaviour of all other linearized coefficients is obtained similarly. This 
finishes the proof of Proposition \ref{lem-lin-eq}.
\end{proof}

\end{document}